\documentclass[final,12pt]{models}

 \usepackage{graphicx}%

\usepackage{amssymb}
 \usepackage{amsthm}

\usepackage{blkarray}

\usepackage{mathtools,amsmath,amssymb,amsthm,bbm}
\usepackage{subfig}
\usepackage{booktabs}
\usepackage{multirow}

\newcommand{\Cov}[2]{\mathbb{C}\text{ov}\left( #1, #2 \right)}

\newcommand{\Var}{\mathbb{V}\text{ar}}

\newcommand{\norm}[1]{\left\lVert #1 \right\rVert}

\newcommand{\1}[1]{\mathbbm{1}_{#1}}

\DeclareMathOperator*{\argmin}{arg\;min}

\DeclareMathOperator*{\argmax}{arg\;max}

\newcommand{\Besov}[3]{{\cal B}_{#1,#2}^{#3}}

\newcommand{\innerp}[2]{ \langle{#1} , {#2}\rangle }

\newcommand{\PP}{\ensuremath{{\mathbb P}}}

\newtheorem{definition}{Definition}
\newtheorem{lemma}{Lemma}
\newtheorem{theorem}{Theorem}
\newdefinition{remark}{Remark}

\usepackage[colorlinks=true]{hyperref}

\journal{Applied and Computational Harmonic Analysis}

\usepackage{fullpage}

\begin{document}
\begin{frontmatter}
\title{Multichannel Deconvolution with Long Range Dependence: Upper bounds on the $L^p$-risk $(1 \le p < \infty)$}

    \author[rafal]{Rafal Kulik}
    \ead{rkulik@uottawa.ca}
    \address[rafal]{Department of Mathematics and Statistics, University of Ottawa, 585 King Edward Avenue, Ottawa ON K1N 6N5, Canada}
    \author[fanis]{Theofanis Sapatinas}
    \ead{{fanis@ucy.ac.cy}}
    \address[fanis]{Department of Mathematics and Statistics,University of Cyprus, P.O. Box 20537,  CY 1678 Nicosia, Cyprus}
    \author[justin]{Justin Rory Wishart}
    \ead{j.wishart@unsw.edu.au}
    \address[justin]{Department of Mathematics and Statistics,University of New South Wales,
    Sydney, NSW  2052,Australia}

    \begin{abstract}
        We consider multichannel deconvolution in a periodic setting with long-memory errors under three different scenarios for the convolution operators, i.e., super-smooth, regular-smooth and box-car convolutions. We investigate global performances of linear and hard-thresholded non-linear wavelet estimators for functions over a wide range of Besov spaces and for a variety of loss functions defining the risk. In particular, we obtain upper bounds on convergence rates using the $L^p$-risk $(1 \le p < \infty)$. Contrary to the case where the errors follow independent Brownian motions, it is demonstrated that multichannel deconvolution with errors that follow independent fractional Brownian motions with different Hurst parameters results in a much more involved situation. An extensive finite-sample numerical study is performed to supplement the theoretical findings.

        \vspace{2mm}


    \end{abstract}
\end{frontmatter}

\section{Introduction}
\label{sec:intro} We study multichannel deconvolution with errors following independent fractional Brownian motions (fBms). More specifically, consider the problem of recovering $f(\cdot) \in L^2(T)$, $T=[0,1]$, on the basis of observing the following noisy convolutions, with known blurring functions $g_\ell(\cdot)$,
\begin{equation}
    dY_{\ell}(t) = K_\ell f(t)dt + \frac{\sigma_\ell }{n^{\alpha_\ell/2}} dB_{H_\ell}
    (t),\ \ \ t \in T, \quad \ell=1,2,\ldots,M, \label{eq:multchan-lm}
\end{equation}
where $\sigma_\ell$ are known positive constants and the convolution operators $K_\ell$ are defined as
\begin{equation}
    K_\ell f(t):=f*g_\ell(t) = \int_T g_\ell(t-x)f(x) dx,\ \ \ t \in T,  \quad
    \ell=1,2,\ldots,M. \label{eq:conv2}
\end{equation}
Here, $B_{H_\ell}(\cdot)$ are independent standard fBms with {\em Hurst} parameters $H_\ell=1-\alpha_\ell/2 \in [1/2,1)$, $\ell=1,2,\ldots,M$; that is, for each $\ell=1,2,\ldots,M$; $B_{H_\ell}(\cdot)$ is a Gaussian process with zero mean and covariance function
\[
 \mathbb{E} (B_{H_\ell}(s)B_{H_\ell}(t))  = \frac{1}{2} \big(|s|^{2H_\ell}+|t|^{2H_\ell}-|t-s|^{2H_\ell}\big), \quad
    s,t \in T, \quad  \ell=1,2,\ldots,M.
\]
The case where $M=1$ corresponds to the fractional Gaussian noise model that can also be viewed as an approximation to the nonparametric regression model with long-range dependence (LRD) (cf. \cite{Wang-1996,Wang-1997}). On the other hand, the case $H_\ell=1/2$, $\ell=1,\ldots,M$; becomes the {\it multichannel} deconvolution with independent standard Brownian motion errors. This model has received attention in studies by \cite{DeCanditiis-Pensky-2006,Pensky-Sapatinas-2009,Pensky-Sapatinas-2010} and \cite{Pensky-Sapatinas-2011}.

We consider the following scenarios for the convolution operators $K_\ell$, $\ell=1,2,\ldots,M$; given by \eqref{eq:conv2} in the Fourier domain where $\widetilde f(m) \coloneqq \int_\mathbb{R} e^{-2\pi i m x} f(x) \, dx$.
\begin{enumerate}
    \item {\em Smooth} convolutions such that, in the Fourier domain,
  \begin{equation}
    \label{eq:K.smooth.M}
    |\widetilde{K_\ell f}(m)| \asymp\,
     |m|^{-\nu_\ell} \exp{ \left\{ - \theta_\ell |m|^{\beta_\ell}\right\} }\,|\widetilde{f}(m)|,
  \end{equation}
  where $m \in \mathbb{R}$, $\ell=1,2,\ldots,M;$ $\beta_\ell > 0$ and $\theta_\ell \ge 0$. In particular, $\nu_\ell \in \mathbb{R}$ if $\theta_\ell > 0$ and $\nu_\ell > 0$ if $\theta_\ell = 0$. The key parameter is $\theta_\ell$, controlling the severity of the decay. The so-called super-smooth deconvolution or exponential decay occurs when $\theta_\ell > 0$ and the regular-smooth or polynomial case occurs when $\theta_\ell = 0$. In the regular-smooth case, each $\nu_\ell >0$ corresponds to the so-called {\em degree of ill-posedness} (DIP) index with $\nu_\ell=0$ representing the {\em direct} (or {\em well-posed}) case.
  \item {\em Box-car} convolutions such that, in the Fourier domain,
    \begin{equation}
        \label{eq:K.box.M} |\widetilde{K_\ell f}(m)|=\frac{\sin (\pi m c_\ell)}{\pi m c_\ell}\,|\widetilde{f}(m)|, \quad m \in \mathbb{R}, \quad
        \ell=1,2,\ldots,M;
    \end{equation}
    where $c_\ell >0$ for each $\ell=1,2,\ldots,M$.
\end{enumerate}
Deconvolution is a common problem in many areas of signal and image processing which include, for instance, light detection and ranging (LIDAR) remote sensing and reconstruction of blurred images. LIDAR is a laser device which emits pulses, reflections of which are gathered by a telescope aligned with the laser. The return signal is used to determine the distance and the position of the reflecting material. However, if the system response function of the LIDAR is longer than the time resolution interval, then the measured LIDAR signal is blurred and the effective accuracy of the LIDAR decreases. This loss of precision can be corrected by deconvolution. In practice, measured LIDAR signals are corrupted by additional noise which renders direct deconvolution impossible. Moreover, if $M\geq2$ (finite) LIDAR devices are used to recover a signal, then we talk about a {\em multichannel} deconvolution problem. The case where $M \geq 2$ in \eqref{eq:multchan-lm}--\eqref{eq:conv2} and $H_\ell=1/2$, $\ell=1,\ldots,M$; i.e., the problem of considering systems of convolution equations with independent errors, was first considered by \cite{Casey-Walnut-1994} in order to evade the ill-posedness of the standard deconvolution model.

In the standard Brownian motion error case, a statistical use of the above idea was investigated by \cite{DeCanditiis-Pensky-2004:JRSSB-Dis,DeCanditiis-Pensky-2006} who proposed adaptive wavelet thresholding estimators. In particular, if $K_\ell$ are regular-smooth convolutions, they showed that an {\em adaptive} wavelet thresholding estimator based on the output from the $M$ channels ``picks'' the convergence rate according to ``the best'' operator $K_\ell$, i.e., the one with the smallest $\nu_\ell$, $\ell=1,2,\ldots,M$. Consequently, adding more channels does not improve the convergence rate of the suggested estimator. On the other hand, if $K_\ell$, $\ell=1,2,\ldots,M$; are box-car convolutions, they showed that adding new channels improves the convergence rate. To be more specific, \cite{DeCanditiis-Pensky-2006} showed, in particular, that the true signal $f(\cdot)$ can be recovered  with accuracy (within a logarithmic factor),
\[
    n^{-2s/(2s+2\nu+1)} \qquad \text{and} \qquad n^{-2s/(2s+(2M+1)/M+1)},
\]
in the regular-smooth and box-car convolutions, respectively. Here, $s>0$ is the smoothness of the underlying signal, $\nu=\min\{\nu_1,\ldots,\nu_M\}$ and the {\em accuracy of estimation} is measured with respect to an upper bound on the $L_2$-risk. In \cite{DeCanditiis-Pensky-2006} the authors did not consider the super-smooth convolutions.

However, real data do not always meet the independence assumption and scientist in diverse fields have observed empirically that correlations between observations that are far apart decay to zero at a slower rate than one would expect from independent data (or, in more general situation, where one deals with short-range dependent data). These fields include astronomy, agronomy, economics chemistry, etc. (see, e.g., \cite{Beran-et-al-2013}).

Therefore, our aim is to study the multichannel deconvolution with errors following fBms. In fact, we show that the situation in this case is much more involved than in the case where the errors follow standard Brownian motions. In particular, we show that in multichannel deconvolution with errors following fBms, the true signal $f(\cdot)$ can be recovered with respect to an upper bound on the $L^p$-risk ($1 \le p < \infty$) with accuracy,
\[
    n^{-s\alpha_{\ell_*}p/(2s+2\nu_{*}+1)}, \qquad (\log n)^{-ps^*/\beta_{\ell_*}} \qquad \text{and} \qquad n^{-s\alpha_*p/(2s+2\widetilde \nu_{*}+1)}
\]
for regular-smooth, super-smooth and box-car deconvolutions respectively (the regular smooth and box-car scenarios are within a logarithmic factor). The parameters in the case of smooth (both regular-smooth and super-smooth) convolutions are defined with
\begin{equation}
    \label{eq:optimal.smooth.channel}\ell_* \coloneqq \argmin_{1 \le \ell \le M} n^{-\alpha_\ell}2^{(\alpha_\ell + 2\nu_\ell)}e^{2\theta_\ell 2^{\beta_\ell}}.
\end{equation}
for the optimal channel and $\nu_{*}$ is defined for the regular-smooth case as
\begin{equation}
    \label{eq:smooth.optimal.nu}
    \nu_*\coloneqq \nu_{\ell_*}+\frac{\alpha_{\ell_*}}{2}-\frac{1}{2}.
\end{equation}
For the case of box-car convolutions the parameters are defined with
\begin{gather}
    \label{eq:box.car.alpha} \alpha_*\coloneqq \min\{\alpha_1,\ldots,\alpha_M\},\quad \text{and}\quad \alpha^* \coloneqq \max\{\alpha_1,\ldots,\alpha_M\},\\
    \label{eq:box.car.optimal.nu}\widetilde \nu_*\coloneqq \frac{2M+1}{2M}+\frac{\alpha^*}{2}-\frac{1}{2}.
\end{gather}

Consequently, the conclusions of \cite{DeCanditiis-Pensky-2006} are
no longer valid here. Even in case of $M=2$, there are different
possibilities for the {\it best scenario}, depending on a
complicated relationship between $s$, $M$, $\nu_\ell$, $\alpha_\ell$, $\theta_\ell$ and $\beta_\ell$, as we illustrate in Section \ref{sec:4}.

\subsection{Modification of the {\tt WaveD} method}

Along with theoretical results, a comparison with the existing {\tt WaveD} method is presented to examine the effect of LRD and multiple channels. Let us compare our modification of the {\tt WaveD} to the standard {\tt R}-package {\tt WaveD} of \cite{Raimondo-Stewart-2007}. In particular, the four signals, {\tt LIDAR}, {\tt Doppler}, {\tt Bumps} and {\tt Blocks} are used as candidate signals in estimation.

For mild levels of LRD ($1/2 < \alpha < 1$) there is not too much difference between the both approaches. However, an improvement is visible for a stronger dependence ($0 < \alpha < 1/2$), as illustrated on \autoref{fig:1} and \autoref{fig:2}. For the parameters $\alpha=0.5$, $\nu=0.5$ and $M = 2$, in the third row, a signal is reconstructed using the proposed multichannel method while the fourth row shows the standard {\tt WaveD} approach using the best channel.

Clearly, the standard {\tt WaveD} approach does not remove artificial noise, which is due to LRD (cf. \autoref{fig:1}). We modify the {\tt WaveD} approach and achieve more reliable estimation by appropriately modified tuning parameters and also truncating the wavelet expansion at an appropriate lower scale level. This truncation is particularly important when there is severe LRD but does not universally yield better estimates (cf. \autoref{fig:2}) and is discussed in more depth in the numerical section later.
\begin{figure}
    \centering
    \begin{minipage}{\columnwidth}
        \subfloat[Doppler signal]{\label{Figure1a}
        \includegraphics[height=0.24\textwidth]{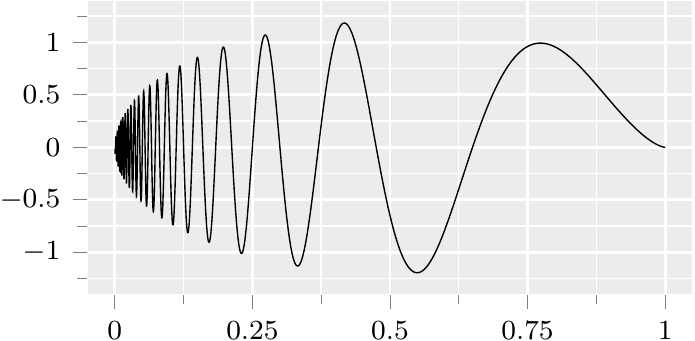}
        }
        \subfloat[LIDAR signal]{\label{Figure1b}
        \includegraphics[height=0.24\textwidth]{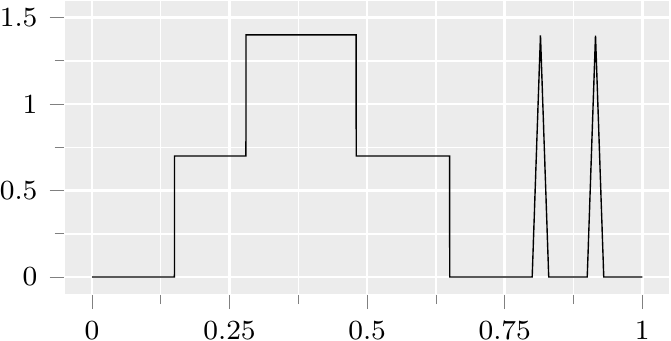}
        }
    \end{minipage}
    \begin{minipage}{\columnwidth}
        \subfloat[Doppler blurred and noisy]{\label{Figure1c}
        \includegraphics[height=0.24\textwidth]{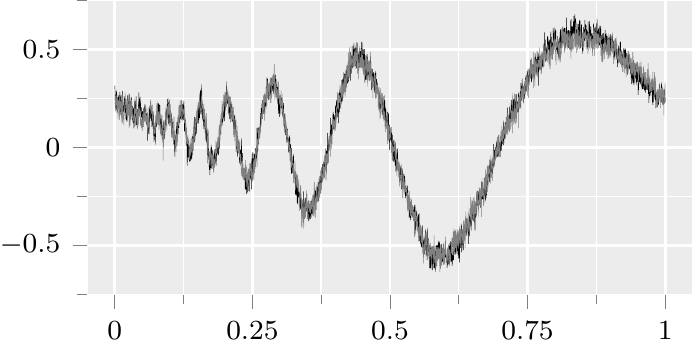}
        }
        \subfloat[LIDAR blurred and noisy]{\label{Figure1d}
        \includegraphics[height=0.24\textwidth]{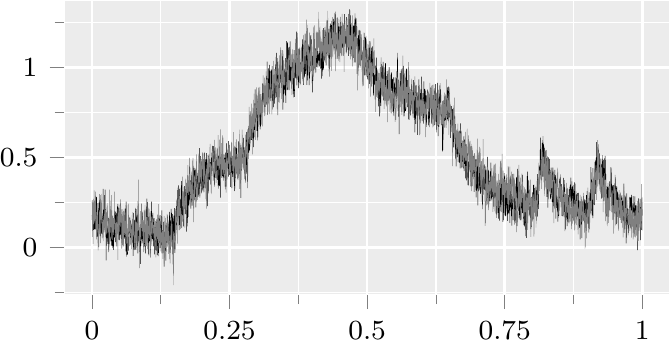}
        }
    \end{minipage}
    \begin{minipage}{\columnwidth}
        \subfloat[Doppler reconstruction]{\label{Figure1e}
        \includegraphics[height=0.24\textwidth]{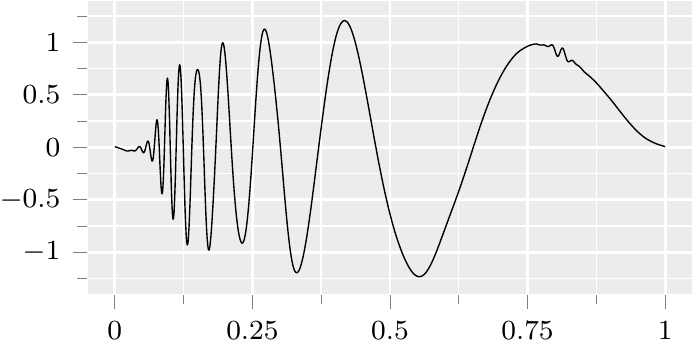}
        }
        \subfloat[LIDAR reconstruction]{\label{Figure1f}
        \includegraphics[height=0.24\textwidth]{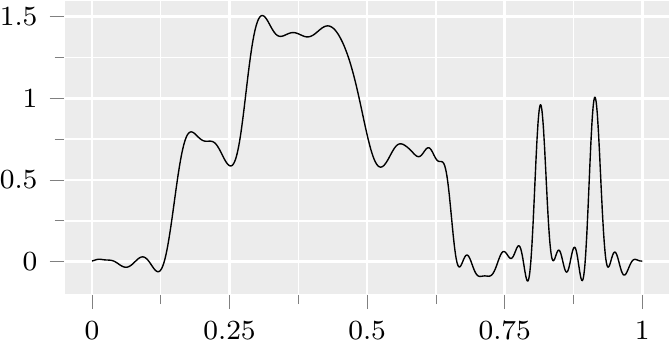}
        }
    \end{minipage}
    \begin{minipage}{\columnwidth}
        \subfloat[Doppler {\tt WaveD} reconstruction]{\label{Figure1g}
        \includegraphics[height=0.24\textwidth]{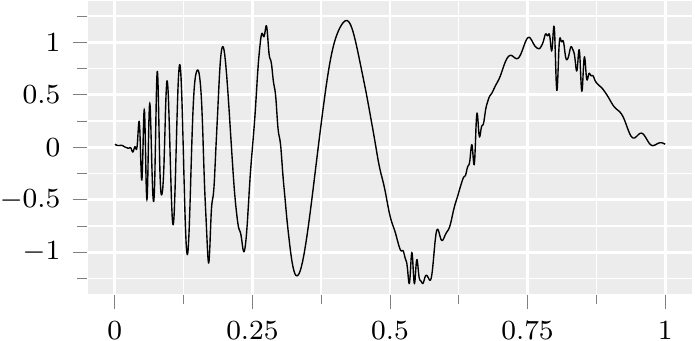}
        }
        \subfloat[LIDAR {\tt WaveD} reconstruction]{\label{Figure1h}
        \includegraphics[height=0.24\textwidth]{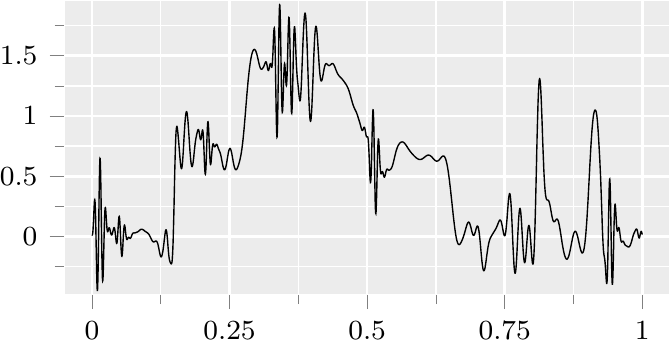}
        }
    \end{minipage}
\caption{\small Top row: original Doppler and LIDAR signal; 2nd row: corresponding blurred and noisy signals, $\nu=0.5$, $\alpha=0.5$ (black line: first channel; grey line: second 
channel); 3rd row: reconstructed signal using the proposed method with $M = 2$ channels; 4th row: reconstructed signal using the standard R-package {\tt WaveD} using the best channel 
(see \eqref{eq:ell.star.estimator}, for the notion of `best channel').}
    \label{fig:1}
\end{figure}

\begin{figure}
    \centering
    \begin{minipage}{\columnwidth}
        \subfloat[Bumps signal]{\label{Figure2a}
        \includegraphics[height=0.24\textwidth]{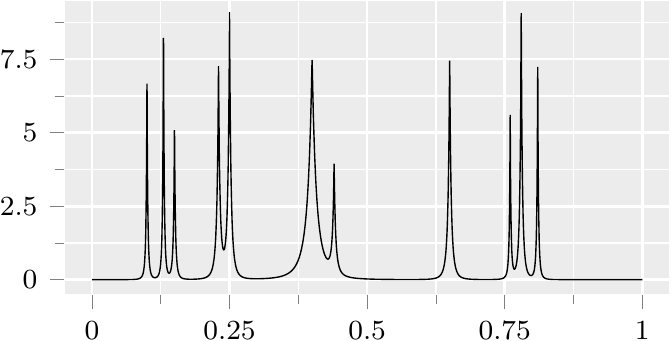}
        }
        \subfloat[Blocks signal]{\label{Figure2b}
        \includegraphics[height=0.24\textwidth]{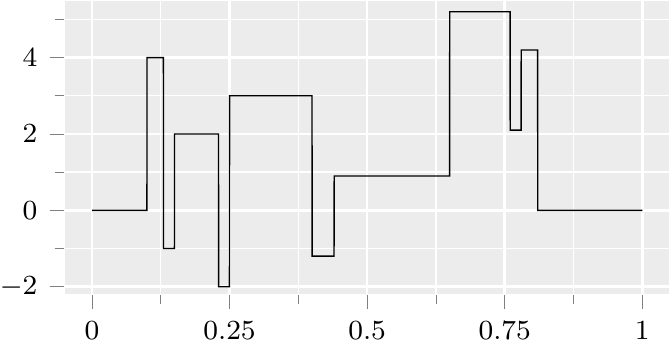}
        }
    \end{minipage}
    \begin{minipage}{\columnwidth}
        \subfloat[Bumps blurred and noisy]{\label{Figure2c}
        \includegraphics[height=0.24\textwidth]{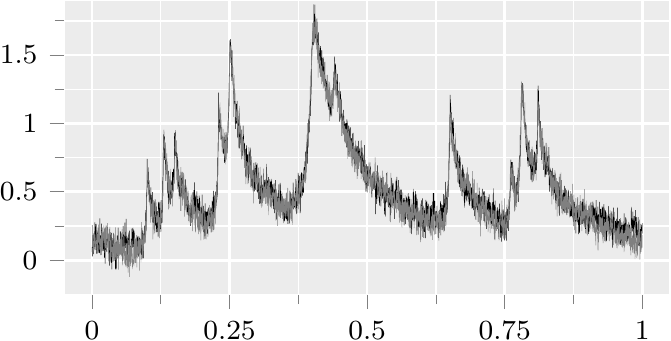}
        }
        \subfloat[Blocks blurred and noisy]{\label{Figure2d}
        \includegraphics[height=0.24\textwidth]{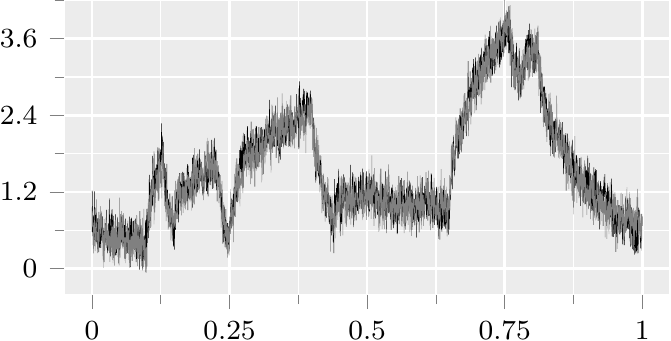}
        }
    \end{minipage}
    \begin{minipage}{\columnwidth}
        \subfloat[Bumps reconstruction]{\label{Figure2e}
        \includegraphics[height=0.24\textwidth]{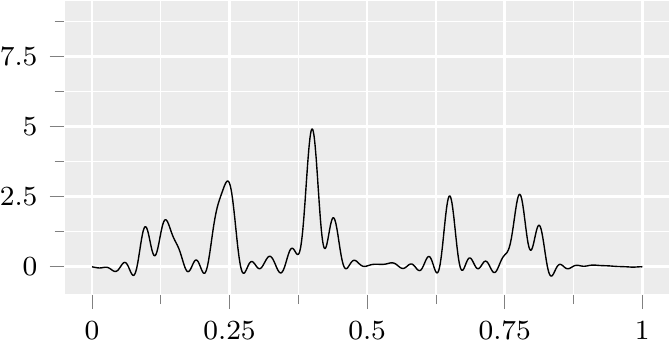}
        }
        \subfloat[Blocks reconstruction]{\label{Figure2f}
        \includegraphics[height=0.24\textwidth]{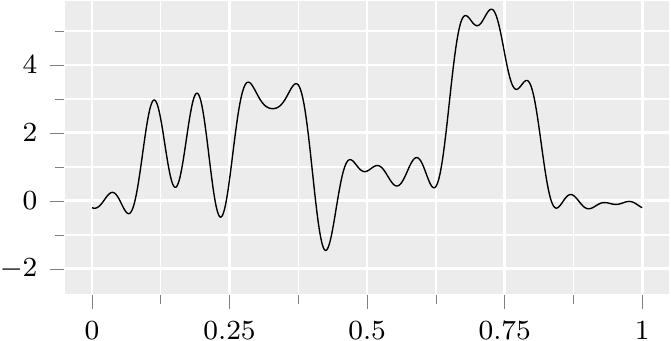}
        }
    \end{minipage}
    \begin{minipage}{\columnwidth}
        \subfloat[Bumps {\tt WaveD} reconstruction]{\label{Figure2g}
        \includegraphics[height=0.24\textwidth]{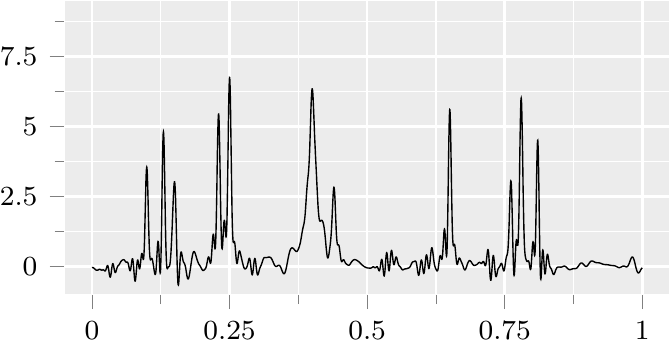}
        }
        \subfloat[Blocks {\tt WaveD} reconstruction]{\label{Figure2h}
        \includegraphics[height=0.24\textwidth]{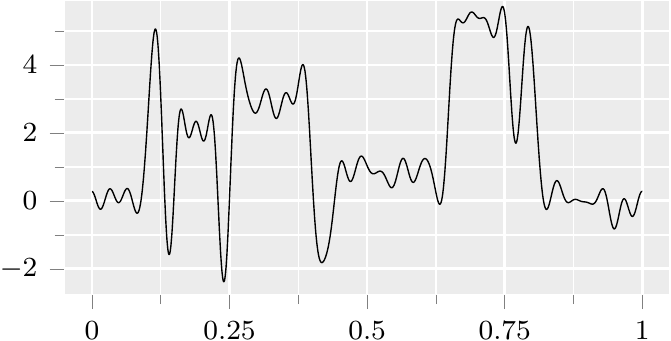}
        }
    \end{minipage}
\caption{\small Top row: original Bumps and Blocks signal; 2nd row: corresponding blurred and noisy signals, $\nu=0.5$, $\alpha=0.5$ (black line: first channel; grey line: second channel) ; 3rd row: reconstructed signal using the proposed method with $M = 2$ channels; 4th row: reconstructed signal using the standard R-package {\tt WaveD} using the best channel (see \eqref{eq:ell.star.estimator}, for the notion of `best channel')}
    \label{fig:2}
\end{figure}

\subsection{Related works}

The case where $M=1$ and $H_1=1/2$ in \eqref{eq:multchan-lm}--\eqref{eq:conv2}
refers to the so-called {\em standard deconvolution model} which attracted attention of a number of researchers.
(Note that the standard deconvolution model is typically {\em ill-posed} in the sense of Hadamard: the inversion does not depend continuously on the observed data, i.e., small noise in the convolved signal leads to a significant error in the estimation procedure.) After a rather rapid progress in this problem in late eighties--early nineties, authors turned to adaptive wavelet solutions of the problem that are optimal (in the {\em minimax} or the {\em maxiset} sense), or near-optimal within a logarithmic factor, in a wide range of Besov balls and for a variety of loss functions defining the risk, and under mild conditions on the blurring function (see, e.g., \cite{Donoho-1995,Abramovich-Silverman-1998,Kalifa-Mallat-2003,Johnstone-et-al-2004,Donoho-Raimondo-2004, Johnstone-Raimondo-2004, Neelamani-et-al-2004,Kerkyacharian-et-al-2007}).

The case $M=1$ and $H_\ell>1/2$ (i.e., standard deconvolution with LRD errors) has been investigated
in \cite{Wang-1996,Wang-1997,Kulik-Raimondo-2009} and \cite{Wishart-2013}.

The case where $\alpha_\ell=1$ for each $\ell=1,2,\ldots,M$; (i.e., the case where in the multichannel deconvolution model \eqref{eq:multchan-lm} the errors follow independent standard Brownian motions) was first considered in \cite{DeCanditiis-Pensky-2006} (extending the results obtained in \cite{Johnstone-et-al-2004} for the case $M=1$).

The case of the multichannel deconvolution with errors following LRD sequences was investigated in \cite{benhaddou:etal} using the minimax approach, extending results obtained in \cite{Pensky-Sapatinas-2009,Pensky-Sapatinas-2010} and \cite{Pensky-Sapatinas-2011}.

The case of nonparametric density estimation for the errors-in-variables problem with LRD has been studied by \cite{Kulik-2008}. In particular, it was shown that LRD has no impact on the optimal convergence properties in the super-smooth scenario. We show similar results for the multichannel deconvolution model presented here.

Finally, for more information regarding the LIDAR device, the reader is referred to, e.g., \cite{Park-et-al-1997} and \cite{Harsdorf-Reuter-2000}.

\subsection{Structure of the paper}
The paper is organised as follows. \autoref{sec:2} contains some preliminaries on the periodised Meyer wavelets and Besov spaces on the unit interval $T$. \autoref{sec:3} provides the construction of the proposed adaptive wavelet thresholding estimators while \autoref{sec:4} contains the corresponding upper bound results over a wide range of Besov spaces and for a variety of loss functions defining the risk, for regular-smooth, super-smooth and box-car convolutions. An extensive simulation study to supplement the theoretical findings of \autoref{sec:4} is performed in \autoref{sec:5}. Conclusions and discussion are given in \autoref{sec:6} and the proofs of the theoretical results and auxiliary results given in \autoref{sec:7} and \hyperref[sec:appendix]{Appendix A}.

\section{Preliminaries}
\label{sec:2}

\subsection{Periodised Meyer wavelets and Besov spaces on the unit interval}
\label{sec:besov}

To avoid edge problems and unnecessary technicalities arising in defining wavelet basis on the unit interval $T$, we will assume that $f(\cdot)$ and $g_\ell(\cdot)$, $\ell=1,2,\ldots,M$; are  periodic on $T$. Moreover, not only for theoretical reasons but also for practical convenience (see, e.g., \cite{Johnstone-et-al-2004}, Sections 2.3, 3.1--3.2), we use band-limited wavelet basis, and in particular the periodised Meyer wavelet basis for which fast algorithms exist (see, e.g., \cite{Kolaczyk-1994} and \cite{Donoho-Raimondo-2004}). Specifically, let $\phi(\cdot)$ and $\psi(\cdot)$ be the Meyer scaling and mother wavelet functions, respectively, on the real line $\mathbb{R}=(-\infty,\infty)$ (see, e.g., \cite{Meyer-1992} or \cite{Mallat-1999}). As usual,
\[
    \phi_{j,k}(t) = 2^{j/2}\phi(2^jt-k), \quad  \psi_{j,k}(t) =
    2^{j/2}\psi(2^jt-k), \quad j\geq 0,\;\;k \in \mathbb{Z}, \quad t \in \mathbb{R},
\]
are, respectively, the dilated and translated Meyer scaling and wavelet functions at resolution level $j$ and scale position $k/2^j$. Similarly to Section 2.3 in \cite{Johnstone-et-al-2004}, we obtain a periodised version of Meyer wavelet basis by periodising the basis functions $\{\phi(\cdot),\psi(\cdot)\}$ on $\mathbb{R}$, i.e., for $j \geq 0$ and $k=0,1,\ldots,2^{j}-1$,
\[
    \Phi_{j,k}(t) = \sum_{i \in \mathbb{Z}} 2^{j/2} \phi(2^j (t +i) -  k),
    \quad \Psi_{j,k}(t) = \sum_{i \in \mathbb{Z}} 2^{j/2} \psi(2^j (t +i) -
    k), \quad t \in T.
\]

In the periodic setting, we recall that Besov spaces are characterised by the behaviour of the wavelet coefficients (see, e.g., \cite{Johnstone-et-al-2004}, Section 2.4), i.e.,
\begin{definition}
    \label{def:SBspace} For  $f(\cdot) \in L^{\pi_0}(T)$, $1\leq \pi_0 < \infty$,
    \begin{equation}
        f(\cdot) \in \Besov{\pi_0}{r}{s}(T)\; \Longleftrightarrow
        \sum_{j=0}^{\infty}2^{j(s+1/2-1/\pi_0)r}\bigg[\sum_{k=0}^{
        2^j-1}|b_{j,k}|^{\pi_0}\bigg]^{r/\pi_0}<\infty, \label{eq:fBesov}
    \end{equation}
with the usual modification if $\pi_0=\infty$ and/or $r=\infty$.
\end{definition}
As usual, the wavelet coefficients $b_{j,k}$ are obtained by $b_{j,k}=\int_T f(t)\psi_{j,k}(t)dt$. The parameter $s >0$ can be thought of as related to the number of derivatives of $f(\cdot)$. With different values of  $\pi_0$ ($1 \leq \pi_0 \leq \infty$) and $r$ ($1 \leq r \leq \infty$), the Besov spaces $\Besov{\pi_0}{r}{s}(T)$ capture a variety of smoothness features in a function including spatially inhomogeneous behaviour.

In the sequel, $\kappa$ will denote the multiple index $(j,k)$ and, adopting standard convention, $\Phi(\cdot) = \Psi_{-1}(\cdot)$, where $\Phi(\cdot)$ corresponds to the periodised scaling function associated with the Meyer wavelet basis mentioned above.

\section{Construction of the adaptive wavelet thresholding and linear estimators} \label{sec:3}

The estimation of $f$ is approached differently for the different deconvolution types. Namely, for regular-smooth and box-car convolutions a wavelet non-linear (hard thresholding) estimator is used while for the super-smooth convolutions a wavelet linear (projection) estimator is used.

To simplify the overall problem, the estimation procedure is considered in the Fourier domain to reduce the convolution operator to a product of Fourier coefficients. Denote the Fourier basis functions, $ e_m(t) \coloneqq e^{2\pi i m t}$, $m \in \mathbb{Z}$, with the corresponding inner product operator, $\innerp{f_1}{f_2} = \int  f_1 (x) \overline{f_2}(x) \, dx$ where $\overline{f}$ denotes the complex conjugate of $f$. Let $h=f*g_{\ell}$. Denote the relevant Fourier coefficients,
\begin{align*}
    \Phi_{m j_0 k} & = \innerp{\Phi_{j_0,k}}{e_m}, \quad \Psi_m^{\kappa} = \Psi_{m j k} = \innerp{\Psi_{j,k}}{e_m},
\end{align*}
\begin{align}
    h_{m,\ell} &= \innerp{h_\ell}{e_m}, \quad y_{m,\ell} = \int_\mathbb{R} \overline{e_{m}}(t) dY_\ell(t),\quad z_{m,\ell} = \int_\mathbb{R} \overline{e_{m}}(t) dB_{H_\ell}(t), \label{eq:Fourier-defs}\\
    f_m  & = \innerp{f}{e_m}, \quad g_{m,\ell} = \innerp{g_\ell}{e_m}, \quad \ell = 1,2,\ldots, M.\nonumber
\end{align}
Applying the Fourier transform to \eqref{eq:multchan-lm}, we get the following sequence space model
\begin{align}
    \label{eq:yl} y_{m,\ell}  & = h_{m,\ell} + \frac{\sigma_\ell}{n^{\alpha_\ell/2}}\; z_{m,l},\quad m \in \mathbb{Z},\quad \ell=1,2,\ldots,M;\\
    \label{eq:hml} h_{m,\ell} & = g_{m,\ell}f_m,\quad m \in \mathbb{Z},\quad \ell=1,2,\ldots,M;
\end{align}
where, for each $\ell$, $\sigma_\ell$ are known positive constants and the structure of the Fourier coefficients, $g_{m,\ell}f_m = \widetilde{K_\ell f}(m)$, is given by \eqref{eq:K.smooth.M} and \eqref{eq:K.box.M} for the smooth-type and box-car convolutions respectively. Following a similar procedure to \cite{DeCanditiis-Pensky-2006}, weights $\gamma_{m,\ell} \overline{g_{m,\ell}}$ are multiplied to the $h_{m,\ell}$ coefficients and added together (where $\gamma_{m,\ell}$ are weights to be specified later). Thus \eqref{eq:hml} leads to the following expression for the target function coefficients,
\[
    f_m=\frac{\sum_{\ell=1}^M \gamma_{m,\ell} \overline{g_{m,\ell}}h_{m,\ell}}{\sum_{\ell=1}^M \gamma_{m,\ell}|g_{m,\ell}|^2}, \quad m \in \mathbb{Z}.
\]
Furthermore using the Parseval identity one can obtain the wavelet coefficients,
\[
    b_{\kappa}=\int_T f(t)\Psi_{\kappa}(t)\,dt=\sum_{m \in \mathbb{Z}}f_{m}\overline{\Psi_{m}^{\kappa}}=\sum_{m \in \mathbb{Z}}\frac{\sum_{\ell=1}^M \gamma_{m,\ell}\overline{g_{m,\ell}}h_{m,\ell}}{\sum_{\ell=1}^M \gamma_{m,\ell} |g_{m,\ell}|^2}\overline{\Psi_{m}^{\kappa}}
\]
which can be estimated using \eqref{eq:yl} with
\begin{equation}
    \label{eq:wav.coeff.estim} \widehat b_{\kappa}=\sum_{m \in C_j}\frac{\sum_{\ell=1}^M \gamma_{m,\ell}\overline{g_{m,\ell}}y_{m,\ell}}{\sum_{\ell=1}^M \gamma_{m,\ell} |g_{m,\ell}|^2}\overline{\Psi_{m}^{\kappa}},
\end{equation}
where $C_j$ denotes the domain of the Meyer wavelet in the Fourier domain,
\begin{equation}
   C_j = \left\{ a \in \mathbb{Z} : \pm a \in \left\{ \left\lceil \frac{2^j}{3} \right\rceil, \left\lceil \frac{2^j}{3} \right\rceil + 1, \ldots, \left\lfloor \frac{2^{j+2}}{3} \right\rfloor \right\} \right\},\label{eq:Cj}
\end{equation} 
where $j \geq 0$. The scaling coefficients $a_{\kappa} = \int_T f(t)\Phi_\kappa(t)\, dt$ and their estimates $\widehat a_{\kappa}$ are defined in a similar manner.

\noindent{\sl Estimators:} A non-linear estimator $\widehat f_n(\cdot)$ of $f(\cdot)$ based on hard thresholding of a wavelet expansion is as follows:
\begin{equation}
    \label{eq:generic} \widehat f_n(t)= \sum_{k = 0}^{2^{j_0}-1}\widehat a_{j_0,k} \Phi_{j_0,k}(t) + \sum_{\kappa \in \Lambda}
    \,\widehat b_{\kappa}\,
    \1{\{|\widehat b_{\kappa}|\geq \lambda\}}\Psi_{\kappa}(t), \quad t \in T,
\end{equation}
where $\1{A}$ denotes the indicator function of the set $A$, the index range, $\Lambda=\Lambda_n$, the coarse scale level $j_0$ and the threshold parameter $\lambda=\lambda_j$ are forthcoming.

A linear (projection) wavelet estimator $\widehat f_n(\cdot)$ of $f(\cdot)$ with coarse scale level $j_0$ is
\begin{equation}
  \label{eq:super-smooth.estimator}\widehat f_n(t) = \sum_{k = 0}^{2^{j_0}-1} \widehat a_{j_0,k} \Phi_{j_0,k}(t).
\end{equation} 

\noindent {\sl Resolution levels:} The range of resolution levels (frequencies) is given by
\[
    \Lambda_n=\{(j,k),\;j_0 \le j\le j_1, \;\; 0 \le k \leq 2^j-1\}.
\]
The coarse scale $j_0$ is defined in the super-smooth case as,
\begin{equation}
  \label{eq:j0.super-smooth} 2^{j_0}\asymp \left( \frac{(\alpha_{\ell_*}-\epsilon)\log n}{2\theta_{\ell_*}}\right)^{1/\beta_{\ell_*}}
\end{equation}
where $\epsilon > 0$ is small, $\theta_\ell$ is the super-smooth parameter defined in \eqref{eq:K.smooth.M} and $\ell_*$ is given by \eqref{eq:optimal.smooth.channel}. For the regular-smooth and box-car case the parameter $j_0$ is not important for the asymptotic convergence of the estimator and we set $j_0 = -1$. The fine scale level $j_1$ is important for the asymptotic convergence results in these cases and is set to be,\begin{equation}
    \label{eq:j1} 2^{j_1}\asymp \left(\frac{n^{\alpha_{\ell_*}}}{\log n}\right)^{1/(2\nu_*+1)}
\end{equation}
for regular-smooth convolutions and
\begin{equation}
    \label{eq:j1.boxcar} 2^{j_1}\asymp \left(\frac{n^{\alpha_*}}{\log n}\right)^{1/(2\widetilde \nu_* + 1)}
\end{equation}
for box-car convolutions, where $\alpha_*$, $\nu_*$, $\widetilde \nu_*$ and $\ell_*$ are defined in \eqref{eq:box.car.alpha}, \eqref{eq:smooth.optimal.nu}, \eqref{eq:box.car.optimal.nu} and \eqref{eq:optimal.smooth.channel} respectively. The fine resolution level $j_1$ in \eqref{eq:j1} coincides with the level given by \cite{Wishart-2013} for the case when $M = 1$, $\nu_1 = \nu$ and $\alpha_1 = \alpha$. 

\medskip

\noindent{\sl Thresholds:} To ease the presentation and include both the regular-smooth and box-car cases, define
\[
   \xi = \begin{cases}
     \alpha_{\ell_*}, & \text{ in the case of regular-smooth deconvolutions};\\
     \alpha_{*}, & \text{ in the case of box-car deconvolutions}.
  \end{cases}
\]
Then the scale level threshold values $\lambda=\lambda_j$ are given by
\begin{equation}
    \label{eq:lambda} \lambda_j=  \zeta\,\tau_{j}\,c_n,
\end{equation}
where the three input parameters are specified as:
\begin{itemize}
    \item $\zeta$: a smoothing parameter, $\zeta>2\sqrt{(p\vee 2)2  \xi}.$
    \item $c_n$: a sample size-dependent scaling factor,
    \begin{equation}
        \label{eq:cn} c_n = \sqrt{\frac{\log n}{n^{\xi}} }.
    \end{equation}
    \item $\tau_j$: a level-dependent scaling factor,
    \begin{align}
        \label{eq:sigma.j} \tau_j^2 &= n^{\xi}\sum_{m \in C_j}  |\Psi^\kappa_m |^2\left( \sum_{\ell=1}^M \sigma_\ell^{-2}n^{\alpha_\ell} |m|^{2H_{\ell}-1}|g_{m,\ell}|^2\right)^{-1}.
    \end{align}
\end{itemize}
In practical applications, the noise levels $\sigma_\ell$; $\ell=1,2,\ldots,M$; are usually unknown. In this case, estimate each $\sigma_\ell$ by $\widehat{\sigma}_\ell$ and define
    \begin{equation}
        \label{eq:sigma.j-hat} \widehat{\tau}_j^2=n^{\xi}\sum_{m \in C_j}  |\Psi^\kappa_m |^2\left( \sum_{\ell=1}^M \widehat \sigma_\ell^{-2}n^{\alpha_\ell} |m|^{2H_{\ell}-1}|g_{m,\ell}|^2\right)^{-1}.
    \end{equation}
This expression is used in the simulation study conducted in Section \ref{sec:4}.
Note that the above thresholds $\lambda_j$ defined in \eqref{eq:lambda} coincide with the ones defined in \cite{DeCanditiis-Pensky-2006} ($M \geq 2$, $\alpha_*=1$).

\section{Upper bound results of the adaptive wavelet thresholding and linear estimators} 
\label{sec:4}
Consider first the smooth convolutions scenario. In this case, the regular-smooth and super-smooth cases are handled when $\theta_\ell = 0$ or $\theta_\ell > 0$ respectively. The super-smooth case is similar to estimating analytic functions with a slow convergence rate. In this scenario linear estimators obtain the optimal (in the minimax sense) convergence rates and hence a linear (projection) wavelet estimator with an appropriate primary resolution level $j_0$ suffices.
\begin{theorem}\label{thm:main-1}
Consider the model described by \eqref{eq:multchan-lm} with $f\in \Besov{\pi_0}{r}{s}(T)$ with $\pi_0 \geq 1$, $s\geq \frac{1}{\pi_0}$. If $\theta_\ell= 0$ for each $\ell = 1,2\ldots,M;$ (regular-smooth case) then consider
\begin{equation}
 \label{eq:Besov-rho}0 < r \leq r_0=\min\Bigg\{
\frac{p(2\nu_*+1)}{2(\nu_*+s)+1},
\frac{(2\nu_*+1)p-2}{2(\nu_*+s)-2/\pi_0+1}  \Bigg\}
\end{equation}
 and the adaptive wavelet estimator $\widehat f_n$ defined in \eqref{eq:generic} with the index range $\Lambda=\Lambda_n$ defined by \eqref{eq:j1} and threshold value $\lambda=\lambda_j$ defined by \eqref{eq:lambda} for some $\zeta>2\sqrt{(p\vee 2)2\alpha_{\ell_*}}$ with $\tau_j$ and $c_n$ given, respectively, by \eqref{eq:sigma.j} and \eqref{eq:cn}. If $\theta_\ell > 0$ for each $\ell = 1,2,\ldots,M;$(super-smooth case) then consider $r >0$ and the linear projection wavelet estimator defined in \eqref{eq:super-smooth.estimator} with coarse scale level, $j_0$, given by \eqref{eq:j0.super-smooth}. Let $p> 1$ be an arbitrary finite real number. Then, there exists a constant $C>0$ such that for all $n \geq 1$,
\[
\mathbb{E} \|\widehat{f}_n-f\|_p^p \le C \left(\frac{\log n}{n^\delta}\right)^{\varrho},
\]
where in the regular-smooth case $\theta_{\ell_*} = 0$ and $\delta = 1$ with
\begin{align}\label{eq:rates-dense-smooth}
\varrho &={\frac{\alpha_{\ell_*} s p}{2(s+(2\nu_*+1)/2)}},& \hbox{ if }& s\geq \frac{(2\nu_*+1)}{2}\Big(\frac{p}{\pi_0}-1\Big);\\
\label{eq:rates-sparse-smooth}
\varrho &=\frac{\alpha_{\ell_*} p(s-1/\pi_0+1/p)}{2(s-1/\pi_0+(2\nu_*+1)/2)},& \hbox{ if }&\frac{1}{\pi_0} - \nu_{*}- \frac{1}{2} \leq s<
\frac{(2\nu_*+1)}{2}\Big(\frac{p}{\pi_0}-1\Big);
\end{align}
while in the super-smooth case, $\theta_{\ell_*} > 0$ and $\delta = 0$ with,
\begin{align}
\varrho &= -ps^*/\beta_{\ell_*} \qquad \text{where } \beta_{\ell_*} > 0\label{eq:rate-super-smooth}
\end{align}
and $s^* = s + 1/p - 1/\min (p,\pi_0)$ and $\ell_*$ is defined with \eqref{eq:optimal.smooth.channel}.

\end{theorem}

Now, consider box-car convolutions scenario. Recall $\alpha_*$ defined by \eqref{eq:box.car.alpha} and $\nu_*$ is now replaced with $\widetilde \nu_*$ defined by \eqref{eq:box.car.optimal.nu}. For the definitions of the `Badly Approximable' (BA) irrational number and the BA irrational tuple that we used in the following statement, see, e.g., p.22 and p.42 of \cite{Schmidt-1980}.

\begin{theorem}\label{thm:main-2}
Consider the model described by \eqref{eq:multchan-lm} and the wavelet estimator $\widehat{f}_n$ defined in \eqref{eq:generic} with the index range $\Lambda=\Lambda_n$ defined by \eqref{eq:j1.boxcar} and threshold value $\lambda=\lambda_j$ defined by \eqref{eq:lambda} for some
$\zeta>2\sqrt{(p\vee 2)2\alpha_*}$ with $\tau_j$ and $c_n$ given, respectively, by \eqref{eq:sigma.j} and \eqref{eq:cn}. Let $p> 1$ be an arbitrary finite real number and assume that one of the $c_1,c_2,\ldots,c_M$ is a BA irrational number and that $c_1,c_2,\ldots,c_M$ ($M \geq 2$) is a BA irrational tuple. If $f\in \Besov{\pi_0}{r}{s}(T)$ with $\pi_0 \geq 1$, $s\geq \frac{1}{\pi_0} - \widetilde \nu_* - 1/2$ and $r$ satisfying \eqref{eq:Besov-rho} with $\nu_*$ replaced with $\widetilde \nu_*$, then, in this case, the result of \autoref{thm:main-1} still holds with $\delta = 1$ and
\begin{align}\label{eq:rates-dense-boxcar}
\varrho&={\frac{\alpha_* s p}{2(s+(2\widetilde \nu_*+1)/2)}},&  \hbox{ if }&
  s\geq \frac{(2\widetilde \nu_*+1)}{2}\Big(\frac{p}{\pi_0}-1\Big);\\
\label{eq:rates-sparse-boxcar}
\varrho &=\frac{\alpha_* p(s-1/\pi_0+1/p)}{2(s-1/\pi_0+(2\widetilde \nu_*+1)/2)},&
 \hbox{ if }& \frac{1}{\pi_0} - \widetilde \nu_{*}- \frac{1}{2}\leq s<
\frac{(2\widetilde \nu_*+1)}{2}\Big(\frac{p}{\pi_0}-1\Big);
\end{align}
where $\widetilde \nu_*$ is defined by \eqref{eq:box.car.optimal.nu} and $\alpha_*$ is defined by \eqref{eq:box.car.alpha}.
\end{theorem}

\begin{remark}There is an {\em elbow effect} or {\em phase transition} in the upper bound on the $L^p$-risk $(1 \leq p < \infty)$ in both the regular-smooth and box-car convolutions. Namely, in the regular-smooth case switching from \eqref{eq:rates-dense-smooth} to \eqref{eq:rates-sparse-smooth} as the assumed smoothness decreases; and similarly switching from \eqref{eq:rates-dense-boxcar} to \eqref{eq:rates-sparse-boxcar} in the box-car case. The two regimes are usually referred to as the {\em dense} and {\em sparse} regions respectively (see \cite{Johnstone-et-al-2004} and \cite{DeCanditiis-Pensky-2006} for the case with independent Brownian motion errors). The upper bound results obtained in \autoref{thm:main-1} for the regular-smooth case and in \autoref{thm:main-2} for the box-car case show that the boundary region of $s$ depends on the LRD indices $\alpha_\ell$, $\ell=1,2,\ldots,M$; and the sparse region is smaller in the case where the errors follow independent fBms.
\end{remark}

\begin{remark}\label{rem:agree}Single Channel, $M = 1$: 
For $\nu_*=\nu_1=0$, the upper bounds obtained on the $L^p$-risk $(1 \le p < \infty)$ in Theorem 1 agree with existing optimal rate results (up to a logarithmic factor) for wavelet regression with long-memory errors obtained by \cite{Wang-1996}, (minimax $L^2$-risk) and \cite{Kulik-Raimondo-2009} (upper bounds on the $L^p$-risk, $1 \le p < \infty$). Similarly, when $\nu_* = \nu_1 > 0$ the results also agree with \cite{Wang-1997} (minimax $L^2$-risk, $p = 2)$ and \cite{Wishart-2013} (upper bounds on the $L^p$-risk, $1 \le p < \infty$). Multichannel, $M > 1$: Our results generalise the results in \cite{DeCanditiis-Pensky-2006} and include the results of their case when $\alpha_* = 1$ (upper bounds on the $L^p$-risk, $1 \le p < \infty$).
\end{remark}

\begin{remark}As expected, the upper bounds deteriorate in the regular-smooth and box-car cases when $\nu_{\ell_*}$ increases (larger DIP) or when $\alpha_{\ell_*}$ decreases (stronger LRD). The combined effect of $\nu_{\ell_*}$ and $\alpha_{\ell_*}$ on the location of the elbow is reverse as the sparse region increases with both $\nu_{\ell_*}$ and $\alpha_{\ell_*}$. Consistent with the literature, the super-smooth case has a logarithmic convergence rate with indices that depend on the underlying smoothness in $s^*$ and the severity of the super-smooth decay in $\beta_{\ell_*}$. The upper bounds on the $L^p$-risk $(1 \le p < \infty)$ in the super-smooth case do not depend on  $\nu_{\ell_*}$ or $\alpha_{\ell_*}$.
\end{remark}

\begin{remark}Our upper bounds on the $L^p$-risk (for $p=2$) are not directly comparable to the maximal upper bounds obtained in \cite{benhaddou:etal}. In that paper the framework is different whereby the number of channels $M$ depends on the number of total observations, $n$, in each channel (i.e. $M=M_n$). However, in our case the number of channels is fixed and not dependent on $n$. Our results are comparable to the works of \cite{DeCanditiis-Pensky-2006,Wishart-2013} demonstrating both the effects of the number of channels and the LRD on the upper bounds on the $L^p$-risk $(1 \le p < \infty)$.
\end{remark}

\section{Simulation study}
\label{sec:5}

A simulation study for $p = 2$ is conducted for the regular-smooth scenario and is heavily based on the algorithm in the  {\tt WaveD} {\tt R}-package  of \cite{Raimondo-Stewart-2007}. In the regular smooth scenario, it is crucial to know $\ell_*$, the `best channel', since it appears in both the smoothing parameter $\zeta$ and in the fine scale level $j_1$. The fine scale parameter is particularly important since it truncates the wavelet expansion early enough to ensure an accurate yet reliable algorithm. Methods have been established for choosing $j_1$ in practice for the single channel regular Brownian motion case by \cite{Cavalier-Raimondo-2007} and expanded to the single channel LRD case by \cite{Wishart-2013}. The method is sketched below and the interested reader is referred to those papers for a more in-depth treatment.

 The method assumes the practitioner can pass the Fourier basis, $f = \left\{ e_u \right\}_{u \in \mathbb{Z}}$, into \eqref{eq:multchan-lm} and denote this new information with,
\[
  d\breve Y_\ell (x) =g_\ell * e_u (x) + \sigma_\ell n^{-\alpha_\ell/2} dB_{H_\ell}(x).
\]
Due to the orthogonality of the Fourier basis, the Fourier domain representation of $\breve Y_\ell$ is
\[
  \breve y_{m,\ell} = \int_\mathbb{R} e_m(x) d\breve Y_\ell (x)  = g_{m, \ell} + \sigma_\ell n^{-\alpha_\ell/2} w_{m,\ell}
\]
where $w_{m,\ell}$ is identically distributed to but independent of $z_{m,\ell}$ (recall \eqref{eq:Fourier-defs} for the definition of $z_{m,\ell}$). Then an estimate of $j_{\ell,1} \asymp \left( n^{\alpha_\ell}/\log n \right)^{1/(\alpha_\ell + 2\nu_\ell)}$ is constructed with,
\[ 
	\widehat j_{\ell,1} = \lfloor \log_2 F_\ell \rfloor -1 
\]
where the stopping time $F_\ell$ is determined in the Fourier domain with $F_\ell = \min\left\{ \omega, \omega > 0 : |\breve y_{\omega,\ell}| \le \omega^{\alpha/2} \varepsilon_\ell \log \varepsilon_\ell^{-2} \right\}$ and $\varepsilon_\ell \coloneqq  \sigma_\ell n^{-\alpha_\ell/2}$. The estimate $\widehat j_{\ell,1}$ is close to $j_{\ell,1}$ with high probability due to Lemma 1 in \cite{Wishart-2013}. 

Then define the overall fine scale estimator with,
\begin{equation}
	\widehat j_1 = \max_{1 \le \ell \le M} \widehat j_{\ell,1}.\label{eq:j1.estimator}
\end{equation}
since the optimal channel defined with
\[
	\ell^* \coloneqq \argmax _{1 \le \ell \le M} \left\{  \left( \frac{n^{\alpha_\ell}}{\log n}\right)^{1/(\alpha_\ell + 2\nu_\ell)} \right\},
\]
is equivalent to the optimal channel $\ell_*$ defined in \eqref{eq:optimal.smooth.channel}. For the same reason, the best channel is estimated as the one with the largest stopping time,
\begin{equation}
  \widehat \ell_* = \argmax_{1 \le \ell \le M } F_\ell \label{eq:ell.star.estimator}
\end{equation}

The theory suggests that the smoothing parameter should satisfy the bound, $\zeta > 4 \sqrt{\alpha_{\ell_*}}$ when $p = 2$. However, as will become evident in the simulations a smaller choice for $\zeta$ results in improved numerical performance. This smaller choice of smoothing parameter compared to the theory is consistent with other numerical results of \cite{Johnstone-et-al-2004} and \cite{Wishart-2013}. The signals used in the simulations are the standard {\tt LIDAR, Doppler, Bumps} and {\tt Blocks} functions that have been used consistently throughout the literature (cf. \cite{Donoho-et-al-1995,Cavalier-Raimondo-2007})

The steps for a simulation study are then as follows:
\begin{enumerate}
\item We choose $f(\cdot)$ to be the {\tt Doppler}, {\tt LIDAR}, {\tt Bumps} or {\tt Blocks} functions.
\item Choose $M$, $n$, $\nu_{\ell}$ and the set of dependence parameters $\alpha_\ell$ for each $\l = 1,2.\ldots, M$; and $n = 2^J$ for $J =12$.
\item Generate $M$ independent FARIMA sequences of length $n$. Each sequence is standardised, to have the same signal-to-noise ratio,
\[
	\text{SNR} = 10 \log_{10} \left( \norm{g_\ell * f}^2/\sigma_\ell^2 \right)
\]
	for three scenarios where $\text{SNR}$ = 10 dB (high noise), 20 dB (medium noise) or 30 dB (low noise). This means that the level of noise compared to the blurred signal is standardised. To simulate the dependent sequence, we use the {\tt R}-package {\tt fracdiff} and the {\tt R}-function {\tt fracdiff.sim}.
\item Estimate the highest permissible scale level, $\widehat j_1$  by using the estimator $\widehat j_1$ defined in \eqref{eq:j1.estimator}.
\item Estimate the `best channel' from the noisy data using \eqref{eq:ell.star.estimator} with $\sigma_\ell$ replaced with $\widehat \sigma_\ell$. Then set the smoothing parameter $\zeta$.
\item Compute $\widehat b_{\kappa}$ using the formula (\ref{eq:wav.coeff.estim}) with level-depending
thresholds $\lambda_j=\zeta \widehat{\tau}_j c_n$ defined in
(\ref{eq:lambda}), where $\widehat{\tau}_j$ and $c_n$ are given
(\ref{eq:sigma.j-hat}) and (\ref{eq:cn}), respectively. The noise level in each channel is estimated using the MAD of the wavelet coefficients at the highest scale level ($J-1$).
\item Compute the above estimates repeatedly to obtain an empirical version of the RMSE with,
\[
	\widehat{MSE}(\widehat f,f) = \widehat{\mathbb{E}} \norm{\widehat f - f}_2 = \frac{1}{m} \sum_{i = 1}^m \norm{\widehat f_i - f}_2
\]
where $m = 1024.$
\end{enumerate}

The results of the simulations are populated in Tables \ref{tab:1} -- \ref{tab:4}

\begin{table}
\centering
\resizebox{\textwidth}{!}{
\begin{tabular}{llllclllclll}
\toprule
\multicolumn{1}{l}{$\nu = 0.3$}&\multicolumn{3}{c}{$\alpha =  1 $}&\multicolumn{1}{c}{}&\multicolumn{3}{c}{$\alpha =  0.8 $}&\multicolumn{1}{c}{}&\multicolumn{3}{c}{$\alpha =  0.6 $}\tabularnewline
\cline{2-4} \cline{6-8} \cline{10-12}
\multicolumn{1}{l}{}&\multicolumn{1}{c}{$M = 1$}&\multicolumn{1}{c}{$M = 2$}&\multicolumn{1}{c}{$M=3$}&\multicolumn{1}{c}{}&\multicolumn{1}{c}{$M = 1$}&\multicolumn{1}{c}{$M = 2$}&\multicolumn{1}{c}{$M=3$}&\multicolumn{1}{c}{}&\multicolumn{1}{c}{$M = 1$}&\multicolumn{1}{c}{$M = 2$}&\multicolumn{1}{c}{$M=3$}\tabularnewline
\midrule
\multicolumn{11}{c}{{\tt LIDAR}: SNR 20 dB}\\ \cmidrule{2-12}
~~$\sqrt{\alpha_{\ell_*}}$&\textbf{0.054}\hfill (7)&\textbf{0.045}\hfill (7)&\textbf{0.041}\hfill (7)&&\textbf{0.064}\hfill (7)&\textbf{0.052}\hfill (7)&\textbf{0.046}\hfill (7)&&\textbf{0.081}\hfill (7)&\textbf{0.064}\hfill (7)&\textbf{0.056}\hfill (7)\tabularnewline
~~$4\sqrt{\alpha_{\ell_*}}$&0.073\hfill (7)&0.06\hfill (7)&0.052\hfill (7)&&0.08\hfill (7)&0.065\hfill (7)&0.057\hfill (7)&&0.093\hfill (7)&0.074\hfill (7)&0.065\hfill (7)\tabularnewline
~~{\small WaveD}&0.06\hfill (7)&0.059\hfill (7)&0.059\hfill (7)&&0.064\hfill (7.6)&0.064\hfill (7.8)&0.064\hfill (7.9)&&0.083\hfill (8)&0.084\hfill (8)&0.084\hfill (8)\tabularnewline
\midrule
\multicolumn{11}{c}{{\tt Doppler}: SNR 20 dB}\\ \cmidrule{2-12}
~~$\sqrt{\alpha_{\ell_*}}$&\textbf{0.039}\hfill (8)&\textbf{0.03}\hfill (8)&\textbf{0.026}\hfill (8)&&0.048\hfill (8)&\textbf{0.036}\hfill (8)&\textbf{0.031}\hfill (8)&&0.062\hfill (8)&\textbf{0.046}\hfill (8)&\textbf{0.039}\hfill (8)\tabularnewline
~~$4\sqrt{\alpha_{\ell_*}}$&0.056\hfill (8)&0.044\hfill (8)&0.036\hfill (8)&&0.059\hfill (8)&0.046\hfill (8)&0.038\hfill (8)&&0.064\hfill (8)&0.05\hfill (8)&0.041\hfill (8)\tabularnewline
~~{\small WaveD}&0.046\hfill (8)&0.045\hfill (8)&0.045\hfill (8)&&\textbf{0.047}\hfill (8)&0.047\hfill (8)&0.047\hfill (8)&&\textbf{0.057}\hfill (8)&0.058\hfill (8)&0.058\hfill (8)\tabularnewline
\midrule
\multicolumn{11}{c}{{\tt Bumps}: SNR 20 dB}\\ \cmidrule{2-12}
~~$\sqrt{\alpha_{\ell_*}}$&\textbf{0.275}\hfill (7)&\textbf{0.27}\hfill (7)&\textbf{0.268}\hfill (7)&&0.28\hfill (7)&0.273\hfill (7)&0.27\hfill (7)&&0.288\hfill (7)&0.278\hfill (7)&0.274\hfill (7)\tabularnewline
~~$4\sqrt{\alpha_{\ell_*}}$&0.279\hfill (7)&0.273\hfill (7)&0.271\hfill (7)&&0.282\hfill (7)&0.276\hfill (7)&0.273\hfill (7)&&0.289\hfill (7)&0.279\hfill (7)&0.276\hfill (7)\tabularnewline
~~{\small WaveD}&0.276\hfill (7)&0.276\hfill (7)&0.275\hfill (7)&&\textbf{0.253}\hfill (7.2)&\textbf{0.231}\hfill (7.4)&\textbf{0.215}\hfill (7.6)&&\textbf{0.189}\hfill (8)&\textbf{0.188}\hfill (8)&\textbf{0.188}\hfill (8)\tabularnewline
\midrule
\multicolumn{11}{c}{{\tt Blocks}: SNR 20 dB}\\ \cmidrule{2-12}
~~$\sqrt{\alpha_{\ell_*}}$&\textbf{0.373}\hfill (6)&\textbf{0.365}\hfill (6)&\textbf{0.363}\hfill (6)&&\textbf{0.384}\hfill (6)&\textbf{0.371}\hfill (6)&\textbf{0.367}\hfill (6)&&0.502\hfill (5)&0.492\hfill (5)&0.489\hfill (5)\tabularnewline
~~$4\sqrt{\alpha_{\ell_*}}$&0.397\hfill (6)&0.373\hfill (6)&0.366\hfill (6)&&0.414\hfill (6)&0.382\hfill (6)&0.372\hfill (6)&&0.508\hfill (5)&0.495\hfill (5)&0.49\hfill (5)\tabularnewline
~~{\small WaveD}&0.376\hfill (6)&0.376\hfill (6)&0.376\hfill (6)&&0.385\hfill (6)&0.385\hfill (6)&0.385\hfill (6)&&\textbf{0.408}\hfill (6)&\textbf{0.408}\hfill (6)&\textbf{0.409}\hfill (6)\tabularnewline
\bottomrule
\end{tabular}}
\caption{RMSE for estimates when $\nu = 0.3$ at mild levels of strong dependence when the number of channels ($M$) increases.}
\label{tab:1}
\end{table}


\begin{table}
\centering
\resizebox{\textwidth}{!}{
\begin{tabular}{llllclllclll}
\toprule
\multicolumn{1}{l}{$\nu = 0.3$}&\multicolumn{3}{c}{$\alpha =  0.5 $}&\multicolumn{1}{c}{}&\multicolumn{3}{c}{$\alpha =  0.3 $}&\multicolumn{1}{c}{}&\multicolumn{3}{c}{$\alpha =  0.1 $}\tabularnewline
\cline{2-4} \cline{6-8} \cline{10-12}
\multicolumn{1}{l}{}&\multicolumn{1}{c}{$M = 1$}&\multicolumn{1}{c}{$M = 2$}&\multicolumn{1}{c}{$M=3$}&\multicolumn{1}{c}{}&\multicolumn{1}{c}{$M = 1$}&\multicolumn{1}{c}{$M = 2$}&\multicolumn{1}{c}{$M=3$}&\multicolumn{1}{c}{}&\multicolumn{1}{c}{$M = 1$}&\multicolumn{1}{c}{$M = 2$}&\multicolumn{1}{c}{$M=3$}\tabularnewline
\midrule
\multicolumn{11}{c}{{\tt LIDAR}: SNR 20 dB}\\ \cmidrule{2-12}
~~$\sqrt{\alpha_{\ell_*}}$&\textbf{0.094}\hfill (7)&\textbf{0.073}\hfill (7)&\textbf{0.063}\hfill (7)&&\textbf{0.115}\hfill (6)&\textbf{0.089}\hfill (6)&\textbf{0.077}\hfill (6)&&0.192\hfill (6)&0.141\hfill (6)&0.118\hfill (6)\tabularnewline
~~$4\sqrt{\alpha_{\ell_*}}$&0.102\hfill (7)&0.081\hfill (7)&0.07\hfill (7)&&0.122\hfill (6)&0.098\hfill (6)&0.084\hfill (6)&&\textbf{0.164}\hfill (6)&\textbf{0.126}\hfill (6)&\textbf{0.107}\hfill (6)\tabularnewline
~~{\small WaveD}&0.103\hfill (8)&0.105\hfill (8)&0.105\hfill (8)&&0.168\hfill (8)&0.169\hfill (8)&0.171\hfill (8)&&0.271\hfill (8)&0.273\hfill (8)&0.273\hfill (8)\tabularnewline
\midrule
\multicolumn{11}{c}{{\tt Doppler}: SNR 20 dB}\\ \cmidrule{2-12}
~~$\sqrt{\alpha_{\ell_*}}$&0.072\hfill (7.7)&0.054\hfill (7.9)&0.045\hfill (8)&&0.091\hfill (7)&0.073\hfill (7)&0.066\hfill (7)&&0.153\hfill (7)&0.113\hfill (7)&0.097\hfill (7)\tabularnewline
~~$4\sqrt{\alpha_{\ell_*}}$&\textbf{0.068}\hfill (7.7)&\textbf{0.053}\hfill (7.9)&\textbf{0.044}\hfill (8)&&\textbf{0.08}\hfill (7)&\textbf{0.065}\hfill (7)&\textbf{0.059}\hfill (7)&&\textbf{0.111}\hfill (7)&\textbf{0.086}\hfill (7)&\textbf{0.076}\hfill (7)\tabularnewline
~~{\small WaveD}&0.069\hfill (8)&0.07\hfill (8)&0.07\hfill (8)&&0.107\hfill (8)&0.109\hfill (8)&0.109\hfill (8)&&0.17\hfill (8)&0.171\hfill (8)&0.172\hfill (8)\tabularnewline
\midrule
\multicolumn{11}{c}{{\tt Bumps}: SNR 20 dB}\\ \cmidrule{2-12}
~~$\sqrt{\alpha_{\ell_*}}$&0.294\hfill (7)&0.281\hfill (7)&0.276\hfill (7)&&0.469\hfill (6)&0.461\hfill (6)&0.458\hfill (6)&&0.496\hfill (6)&0.475\hfill (6)&0.467\hfill (6)\tabularnewline
~~$4\sqrt{\alpha_{\ell_*}}$&0.294\hfill (7)&0.282\hfill (7)&0.278\hfill (7)&&0.467\hfill (6)&0.461\hfill (6)&0.458\hfill (6)&&0.489\hfill (6)&0.472\hfill (6)&0.466\hfill (6)\tabularnewline
~~{\small WaveD}&\textbf{0.201}\hfill (8)&\textbf{0.201}\hfill (8)&\textbf{0.202}\hfill (8)&&\textbf{0.246}\hfill (8)&\textbf{0.247}\hfill (8)&\textbf{0.248}\hfill (8)&&\textbf{0.329}\hfill (8)&\textbf{0.331}\hfill (8)&\textbf{0.33}\hfill (8)\tabularnewline
\midrule
\multicolumn{11}{c}{{\tt Blocks}: SNR 20 dB}\\ \cmidrule{2-12}
~~$\sqrt{\alpha_{\ell_*}}$&0.511\hfill (5)&0.497\hfill (5)&0.492\hfill (5)&&0.82\hfill (4)&0.808\hfill (4)&0.804\hfill (4)&&1.116\hfill (3)&1.094\hfill (3)&1.087\hfill (3)\tabularnewline
~~$4\sqrt{\alpha_{\ell_*}}$&0.518\hfill (5)&0.5\hfill (5)&0.494\hfill (5)&&0.823\hfill (4)&0.809\hfill (4)&0.805\hfill (4)&&1.116\hfill (3)&1.094\hfill (3)&1.087\hfill (3)\tabularnewline
~~{\small WaveD}&\textbf{0.43}\hfill (6)&\textbf{0.43}\hfill (6)&\textbf{0.43}\hfill (6)&&\textbf{0.506}\hfill (6)&\textbf{0.507}\hfill (6)&\textbf{0.507}\hfill (6)&&\textbf{0.68}\hfill (6.1)&\textbf{0.691}\hfill (6.3)&\textbf{0.697}\hfill (6.3)\tabularnewline
\bottomrule
\end{tabular}}
\caption{RMSE for estimates when $\nu = 0.3$ at severe levels of strong dependence when the number of channels ($M$) increases.}
\label{tab:2}
\end{table}

\begin{table}
\centering
\resizebox{\textwidth}{!}{
\begin{tabular}{llllclllclll}
\toprule
\multicolumn{1}{l}{$\nu = 0.5$}&\multicolumn{3}{c}{$\alpha =  1 $}&\multicolumn{1}{c}{}&\multicolumn{3}{c}{$\alpha =  0.8 $}&\multicolumn{1}{c}{}&\multicolumn{3}{c}{$\alpha =  0.6 $}\tabularnewline
\cline{2-4} \cline{6-8} \cline{10-12}
\multicolumn{1}{l}{}&\multicolumn{1}{c}{$M = 1$}&\multicolumn{1}{c}{$M = 2$}&\multicolumn{1}{c}{$M=3$}&\multicolumn{1}{c}{}&\multicolumn{1}{c}{$M = 1$}&\multicolumn{1}{c}{$M = 2$}&\multicolumn{1}{c}{$M=3$}&\multicolumn{1}{c}{}&\multicolumn{1}{c}{$M = 1$}&\multicolumn{1}{c}{$M = 2$}&\multicolumn{1}{c}{$M=3$}\tabularnewline
\midrule
\multicolumn{11}{c}{{\tt LIDAR}: SNR 20 dB}\\ \cmidrule{2-12}
~~$\sqrt{\alpha_{\ell_*}}$&\textbf{0.073}\hfill (6)&\textbf{0.062}\hfill (6)&\textbf{0.056}\hfill (6)&&\textbf{0.085}\hfill (6)&\textbf{0.07}\hfill (6)&\textbf{0.063}\hfill (6)&&\textbf{0.104}\hfill (6)&\textbf{0.085}\hfill (6)&\textbf{0.075}\hfill (6)\tabularnewline
~~$4\sqrt{\alpha_{\ell_*}}$&0.094\hfill (6)&0.081\hfill (6)&0.073\hfill (6)&&0.103\hfill (6)&0.088\hfill (6)&0.08\hfill (6)&&0.122\hfill (6)&0.1\hfill (6)&0.09\hfill (6)\tabularnewline
~~{\small WaveD}&0.083\hfill (6)&0.082\hfill (6)&0.083\hfill (6)&&0.086\hfill (6)&0.086\hfill (6)&0.086\hfill (6)&&0.105\hfill (6.1)&0.106\hfill (6.2)&0.107\hfill (6.2)\tabularnewline
\midrule
\multicolumn{11}{c}{{\tt Doppler}: SNR 20 dB}\\ \cmidrule{2-12}
~~$\sqrt{\alpha_{\ell_*}}$&\textbf{0.059}\hfill (7)&\textbf{0.053}\hfill (7)&\textbf{0.051}\hfill (7)&&\textbf{0.067}\hfill (7)&\textbf{0.058}\hfill (7)&\textbf{0.054}\hfill (7)&&0.084\hfill (6.9)&\textbf{0.067}\hfill (7)&\textbf{0.061}\hfill (7)\tabularnewline
~~$4\sqrt{\alpha_{\ell_*}}$&0.076\hfill (7)&0.061\hfill (7)&0.056\hfill (7)&&0.082\hfill (7)&0.065\hfill (7)&0.059\hfill (7)&&0.092\hfill (6.9)&0.071\hfill (7)&0.063\hfill (7)\tabularnewline
~~{\small WaveD}&0.065\hfill (7)&0.064\hfill (7)&0.064\hfill (7)&&0.067\hfill (7)&0.067\hfill (7)&0.067\hfill (7)&&\textbf{0.078}\hfill (7)&0.078\hfill (7)&0.078\hfill (7)\tabularnewline
\midrule
\multicolumn{11}{c}{{\tt Bumps}: SNR 20 dB}\\ \cmidrule{2-12}
~~$\sqrt{\alpha_{\ell_*}}$&\textbf{0.457}\hfill (6)&\textbf{0.455}\hfill (6)&\textbf{0.453}\hfill (6)&&0.461\hfill (6)&0.457\hfill (6)&0.455\hfill (6)&&0.467\hfill (6)&0.46\hfill (6)&0.458\hfill (6)\tabularnewline
~~$4\sqrt{\alpha_{\ell_*}}$&0.457\hfill (6)&0.456\hfill (6)&0.455\hfill (6)&&0.461\hfill (6)&0.457\hfill (6)&0.456\hfill (6)&&0.467\hfill (6)&0.46\hfill (6)&0.458\hfill (6)\tabularnewline
~~{\small WaveD}&0.457\hfill (6)&0.457\hfill (6)&0.457\hfill (6)&&\textbf{0.441}\hfill (6.1)&\textbf{0.429}\hfill (6.2)&\textbf{0.418}\hfill (6.3)&&\textbf{0.332}\hfill (6.9)&\textbf{0.319}\hfill (7)&\textbf{0.318}\hfill (7)\tabularnewline
\midrule
\multicolumn{11}{c}{{\tt Blocks}: SNR 20 dB}\\ \cmidrule{2-12}
~~$\sqrt{\alpha_{\ell_*}}$&\textbf{0.494}\hfill (5)&\textbf{0.488}\hfill (5)&\textbf{0.486}\hfill (5)&&\textbf{0.505}\hfill (5)&\textbf{0.494}\hfill (5)&\textbf{0.49}\hfill (5)&&0.807\hfill (4)&0.801\hfill (4)&0.8\hfill (4)\tabularnewline
~~$4\sqrt{\alpha_{\ell_*}}$&0.506\hfill (5)&0.493\hfill (5)&0.489\hfill (5)&&0.519\hfill (5)&0.501\hfill (5)&0.494\hfill (5)&&0.811\hfill (4)&0.803\hfill (4)&0.801\hfill (4)\tabularnewline
~~{\small WaveD}&0.497\hfill (5)&0.497\hfill (5)&0.497\hfill (5)&&0.506\hfill (5)&0.506\hfill (5)&0.506\hfill (5)&&\textbf{0.527}\hfill (5)&\textbf{0.528}\hfill (5)&\textbf{0.528}\hfill (5)\tabularnewline
\bottomrule
\end{tabular}}
\caption{RMSE for estimates when $\nu = 0.5$ at mild levels of strong dependence when the number of channels ($M$) increases.}
\label{tab:3}
\end{table}


\begin{table}[htp]
\centering
\resizebox{\textwidth}{!}{
\begin{tabular}{llllclllclll}
\toprule
\multicolumn{1}{l}{$\nu = 0.5$}&\multicolumn{3}{c}{$\alpha =  0.5 $}&\multicolumn{1}{c}{}&\multicolumn{3}{c}{$\alpha =  0.3 $}&\multicolumn{1}{c}{}&\multicolumn{3}{c}{$\alpha =  0.1 $}\tabularnewline
\cline{2-4} \cline{6-8} \cline{10-12}
\multicolumn{1}{l}{}&\multicolumn{1}{c}{$M = 1$}&\multicolumn{1}{c}{$M = 2$}&\multicolumn{1}{c}{$M=3$}&\multicolumn{1}{c}{}&\multicolumn{1}{c}{$M = 1$}&\multicolumn{1}{c}{$M = 2$}&\multicolumn{1}{c}{$M=3$}&\multicolumn{1}{c}{}&\multicolumn{1}{c}{$M = 1$}&\multicolumn{1}{c}{$M = 2$}&\multicolumn{1}{c}{$M=3$}\tabularnewline
\midrule
\multicolumn{11}{c}{{\tt LIDAR}: SNR 20 dB}\\ \cmidrule{2-12}
~~$\sqrt{\alpha_{\ell_*}}$&\textbf{0.11}\hfill (5)&\textbf{0.094}\hfill (5)&\textbf{0.087}\hfill (5)&&\textbf{0.142}\hfill (5)&\textbf{0.115}\hfill (5)&\textbf{0.103}\hfill (5)&&0.215\hfill (4)&\textbf{0.184}\hfill (4)&\textbf{0.172}\hfill (4)\tabularnewline
~~$4\sqrt{\alpha_{\ell_*}}$&0.134\hfill (5)&0.109\hfill (5)&0.097\hfill (5)&&0.163\hfill (5)&0.128\hfill (5)&0.113\hfill (5)&&\textbf{0.214}\hfill (4)&0.185\hfill (4)&0.173\hfill (4)\tabularnewline
~~{\small WaveD}&0.134\hfill (6.5)&0.14\hfill (6.7)&0.142\hfill (6.9)&&0.243\hfill (7)&0.246\hfill (7)&0.246\hfill (7)&&0.399\hfill (7)&0.401\hfill (7)&0.4\hfill (7)\tabularnewline
\midrule
\multicolumn{11}{c}{{\tt Doppler}: SNR 20 dB}\\ \cmidrule{2-12}
~~$\sqrt{\alpha_{\ell_*}}$&0.104\hfill (6)&0.097\hfill (6)&0.094\hfill (6)&&0.122\hfill (6)&0.107\hfill (6)&0.101\hfill (6)&&0.183\hfill (5)&0.166\hfill (5)&0.16\hfill (5)\tabularnewline
~~$4\sqrt{\alpha_{\ell_*}}$&0.106\hfill (6)&0.098\hfill (6)&0.095\hfill (6)&&\textbf{0.115}\hfill (6)&\textbf{0.103}\hfill (6)&\textbf{0.099}\hfill (6)&&\textbf{0.172}\hfill (5)&\textbf{0.161}\hfill (5)&\textbf{0.156}\hfill (5)\tabularnewline
~~{\small WaveD}&\textbf{0.091}\hfill (7)&\textbf{0.092}\hfill (7)&\textbf{0.092}\hfill (7)&&0.14\hfill (7)&0.141\hfill (7)&0.141\hfill (7)&&0.216\hfill (7)&0.218\hfill (7)&0.218\hfill (7)\tabularnewline
\midrule
\multicolumn{11}{c}{{\tt Bumps}: SNR 20 dB}\\ \cmidrule{2-12}
~~$\sqrt{\alpha_{\ell_*}}$&0.688\hfill (5.2)&0.643\hfill (5.3)&0.611\hfill (5.4)&&0.742\hfill (5)&0.736\hfill (5)&0.734\hfill (5)&&0.88\hfill (4)&0.873\hfill (4)&0.871\hfill (4)\tabularnewline
~~$4\sqrt{\alpha_{\ell_*}}$&0.689\hfill (5.2)&0.643\hfill (5.3)&0.611\hfill (5.4)&&0.743\hfill (5)&0.736\hfill (5)&0.734\hfill (5)&&0.88\hfill (4)&0.873\hfill (4)&0.871\hfill (4)\tabularnewline
~~{\small WaveD}&\textbf{0.334}\hfill (7)&\textbf{0.334}\hfill (7)&\textbf{0.334}\hfill (7)&&\textbf{0.387}\hfill (7)&\textbf{0.389}\hfill (7)&\textbf{0.389}\hfill (7)&&\textbf{0.487}\hfill (7)&\textbf{0.488}\hfill (7)&\textbf{0.488}\hfill (7)\tabularnewline
\midrule
\multicolumn{11}{c}{{\tt Blocks}: SNR 20 dB}\\ \cmidrule{2-12}
~~$\sqrt{\alpha_{\ell_*}}$&0.813\hfill (4)&0.805\hfill (4)&0.802\hfill (4)&&1.085\hfill (3)&1.078\hfill (3)&1.075\hfill (3)&&1.126\hfill (3)&1.098\hfill (3)&1.089\hfill (3)\tabularnewline
~~$4\sqrt{\alpha_{\ell_*}}$&0.819\hfill (4)&0.807\hfill (4)&0.803\hfill (4)&&1.086\hfill (3)&1.078\hfill (3)&1.075\hfill (3)&&1.129\hfill (3)&1.099\hfill (3)&1.089\hfill (3)\tabularnewline
~~{\small WaveD}&\textbf{0.548}\hfill (5)&\textbf{0.549}\hfill (5)&\textbf{0.549}\hfill (5)&&\textbf{0.625}\hfill (5)&\textbf{0.626}\hfill (5)&\textbf{0.625}\hfill (5)&&\textbf{0.796}\hfill (5)&\textbf{0.798}\hfill (5)&\textbf{0.8}\hfill (5)\tabularnewline
\bottomrule
\end{tabular}}
\caption{RMSE for estimates when $\nu = 0.5$ at severe levels of strong dependence when the number of channels ($M$) increases.}
\label{tab:4}
\end{table}

\noindent{\bf Comments and analysis}

The numerical study is considered for three particular contexts. Namely, the effect of the severity of LRD, the effect of multiple channels and the degree of ill-posedness. The method is also compared with using the standard {\tt WaveD} estimator on the `best channel' in the sense of the algorithm posed at the start of this Section. The results are contained in Tables \ref{tab:1} -- \ref{tab:4}. Simulations were conducted for a large range of noise levels with SNR = 10,15,20,25 and 30 dB but are omitted due to space constraints. The estimates at other noise levels showed similar results to those displayed here and did not add further to the concepts being discussed below.

Performance of our method (and the {\tt WaveD} method) is reliant on two key steps. The most important step is choosing the fine scale level $j_1$ to truncate the expansion at the highest allowable level before performance deteriorates. A less important but still crucial step is to choose the smoothing parameter $\zeta$ appropriately (the smoothing parameter $\eta$ for the {\tt WaveD} algorithm is fixed at its default of $\sqrt{6}$). 

To demonstrate both the role of $j_1$ and $\zeta$, the RMSE of the estimators in all the forthcoming contexts are presented inside the cells of the tables with the average fine scale level $\widehat j_1$ shown in parenthesis. The values of $\zeta$ are given in the first column (with {\tt WaveD} denoting the standard {\tt WaveD} estimator in the best possible channel).

Theoretical arguments suggest that $\zeta$ should be at least greater than $4 \sqrt{\alpha_{\ell_*}}$ for $p = 2$. Simulations were conducted for more liberal and conservative choices of $\zeta$ with $\zeta \in \left( \sqrt{\alpha_{\ell_*}}, 8\sqrt{\alpha_{\ell_*}} \right)$. In almost all cases, the performance was optimal using the smaller choice of $\zeta = \sqrt{\alpha_{\ell_*}}$. The exceptions generally being when the dependence was considerably strong ($\alpha < 0.3)$ and $M = 1$.

 As is consistent with \cite{Wishart-2013}, allowing higher scales can capture more transient features of a signal but can be at the cost of spurious effects of LRD noise being included. Sometimes early truncation can be beneficial to performance or detrimental to performance based on the features of the signal. For example, the estimation performance on the {\tt LIDAR } and {\tt Doppler} signals benefits from the earlier truncation but is detrimental to the estimation of the {\tt Bumps} and {\tt Blocks} signals. In the latter estimated signals, the captured transient features at higher scales outweigh the potential loss incurred from including spurious LRD noise effects. A potential reason that the {\tt LIDAR} signal is estimated well in the multichannel method in simulations compared to the similar {\tt Blocks} signal is the close proximity of the jumps combined with the early truncation (small $j_1$) in the expansion. The {\tt WaveD} does not truncate early to avoid the LRD effects and hence captures the jumps better (cf. Figures \ref{fig:1} and \ref{fig:2}). Finally in the {\tt Bumps} signal, the {\tt WaveD} method consistently outperformed the multichannel estimator (except with the liberal choice with $\zeta = 1$ when $\alpha_{\ell_*} = 1)$. This makes sense since the captured high frequency local features of the {\tt Bumps} signal used with a larger $j_1$ outweigh the loss incurred by spurious LRD effects. All of the aforementioned points are evident across Tables \ref{tab:1} -- \ref{tab:4} and shown visually as particular cases in Figures \ref{fig:1} and \ref{fig:2}.

Supporting the theory and being consistent with previous results in the literature, as the degree of ill-posedness increases ($\nu$ increases), the performance of estimation deteriorates. This is demonstrated by comparing results from Tables \ref{tab:1} -- \ref{tab:2} with the results in Tables \ref{tab:3} -- \ref{tab:4}. 

 In the same vein, as the level of dependence increases ($\alpha$ decreases), the performance deteriorates. Studying Tables \ref{tab:1} -- \ref{tab:4} in more detail, consider the effect of $\alpha$ while keeping $M$ fixed and $\nu$ fixed. As is consistent with the theoretical upper bound on rates of convergence established in Section \ref{sec:4}, the convergence rate deteriorates as the level of dependence increases ($\alpha$ decreases). 

The theory also suggests that the convergence rate only relies on the best available channel. However, numerically this doesn't seem to be the case. Interestingly, when keeping the dependence and DIP levels fixed across multiple channels, the inclusion of more channels (increasing $M$) generally results in improved estimation performance for the multichannel estimator in all signals while the {\tt WaveD} estimator has the same performance across multiple channels. This should not seem surprising since the {\tt WaveD} estimator is only used the `best channel' meaning only $n = 4096$ observations are being used each time. The multichannel estimator though is using a weighted average of all channels using 4096, 8192 and 12288 observations respectively in the $M = 1$, $2$ and $3$ scenarios.

\section{Conclusion}
\label{sec:6}
In this paper we considered multichannel deconvolution with errors following fractional Brownian motions, with different Hurst parameters. We established upper bounds on the $L^p$-risk $(1 \le p < \infty)$ for the non-linear wavelet estimators for regular-smooth and box-car convolutions and linear wavelet estimator for super-smooth convolutions. In particular, we extended the findings from \cite{DeCanditiis-Pensky-2006} and demonstrated that they are no longer valid in the LRD set-up. That is, in the box-car case adding new channels is beneficial for the upper bound only if the additional channel isn't outweighed by the dependence in the sense of $\widetilde \nu_*$ defined in \eqref{eq:box.car.optimal.nu} and the upper bound in \autoref{thm:main-2}. While in the regular-smooth case, adding new channels might perhaps improve the upper bound. An improved upper bound would arise if the $\alpha$ and DIP parameters in the new channel are better in the sense of \eqref{eq:optimal.smooth.channel}. In both regular-smooth and box-car cases though, LRD affects upper bounds which is consistent with previous findings in \cite{Wang-1997,Kulik-Raimondo-2009} and \cite{Wishart-2013}. In the super-smooth case, adding new channels is also beneficial, however, the upper bounds do not involve LRD.

We supported our theoretical findings by extensive simulations studies for the regular-smooth case using the $L^p$-risk for $p = 2$. We found that adding new channels improves performance, especially for severe levels of LRD. On the other hand, the optimal choice of threshold level was in some instances different than the one suggested by the theory. The optimal choice highly depends on the underlying signal. One has to remember though, that the established theory is \textit{asymptotic} in nature, whereas simulations studies are based on finite sample properties. This explains the aforementioned discrepancy.

A possible direction for future research is to explore and extend our upper bounds to minimax type rates towards the direction of \cite{benhaddou:etal} obtained for the $L^2$-risk in the discrete model when the number of channels, $M$,  also depend on the total number of observations $n$, i.e., $M=M_n$.
\section{Proofs}
\label{sec:7}

We provide technical details for the proofs of Theorems \ref{thm:main-1} and \ref{thm:main-2}. In the regular-smooth and box-car cases, the proofs are based on the maxiset theorem (see Theorem 6.1 in \cite{Kerkyacharian-Picard-2000}). The steps are similar to those of \cite{Johnstone-et-al-2004} and \cite{DeCanditiis-Pensky-2006}, with necessary modifications. In the super-smooth case we do not need the maxiset theorem but proceed according to \cite{Petsa-Sapatinas-2009} and consider the $L^p$-risk ($1 \le p < \infty$) directly.

\subsection{Stochastic analysis of estimated wavelet coefficients}
\label{sec:varnoisecoef}
By definition, it is clearly seen that the estimated wavelet coefficients have no bias.  Consider now the covariance structure of the $z_{\cdot \ell}$ process where $z_{m,\ell} = \int_\mathbb{R} \overline{e_m}(t) dB_{H_\ell}(t)$. It is assumed that, $B_{H_\ell}$ is independent of $B_{H_\ell'}$ for $\ell \neq \ell'$. This has the immediate consequence that, $\Cov{z_{m\ell}}{z_{m'\ell'}} = 0$ for $\ell \neq \ell'$. Using the results of Section 5.2 of \cite{Wishart-2013}, the covariance of the fBm coefficients within each channel is,
\begin{equation}
    \Cov{ z_{m\ell}}{ z_{m'\ell}} = |m m'|^{1/2 - {H_\ell}} \sum_{\kappa' } \psi^{\kappa'}_{m} \overline{\psi^{\kappa'}_{m'}},\label{eq:exactcov}
\end{equation}
where $\psi$ is the Meyer wavelet and $\kappa' = (j',k')$.

The result in \eqref{eq:exactcov} would seem to imply that the covariance matrix of $z_{m\ell}$ is non-trivial. However, applying \autoref{lem:dyadic}, the covariance matrix reduces to
\begin{equation}
  \Cov{z_{m\ell}}{ z_{m'\ell}} = |mm'|^{1/2 -H_\ell}\sum_{j \in \mathbb{Z}}\1{\left\{ \frac{m-m'}{2^j} \in \mathbb{Z} \right\}}\psi_{m2^{-j}} \overline{\psi_{m'2^{-j}}}.\label{eq:covZHsum}
\end{equation}
Thus we are in a position to bound the variance of the estimated wavelet coefficients (recall $\gamma_{m,\ell}$ are weighting
constants),
\begin{align}
    \Var \left(\widehat b_{\kappa} \right) & = \Var \left(b_{\kappa} +\sum_{m \in C_j} \sum_{\ell = 1}^M  \frac{\gamma_{m,\ell} n^{-\alpha_\ell/2}\sigma_\ell \overline{g_{m,\ell}}z_{m\ell}}{ \sum_{\ell=1}^M \gamma_{m,\ell}
|g_{m,\ell}|^2}\overline{\Psi^\kappa_{m}} , \right)\nonumber\\
                                               & = \sum_{\ell = 1}^M  \sum_{m,m' \in C_j}\frac{\gamma_{m,\ell}\gamma_{m',\ell}\sigma_\ell^2 n^{-\alpha_\ell} |mm'|^{1/2 -H_\ell}\overline{g_{m,\ell}}g_{m',\ell} \overline{\Psi^\kappa_{m}}{\Psi^\kappa_{m}}}{\left(  \sum_{\ell=1}^M \gamma_{m,\ell}|g_{m,\ell}|^2\right)\left(  \sum_{\ell=1}^M \gamma_{m',\ell}
|g_{m',\ell}|^2\right)} \nonumber\\
&\qquad \qquad \times \sum_{j' \in \mathbb{Z}}\1{\left\{ \frac{m-m'}{2^j} \in \mathbb{Z} \right\}}\psi_{m2^{-j'}} \overline{\psi_{m'2^{-j'}}},\label{eq:varBetaK}
\end{align}
where the second last line follows by \eqref{eq:covZHsum} and the independence of the fBms. Apply \autoref{lem:identity-covariance} to \eqref{eq:varBetaK} yields,
\begin{equation}
	\Var \left(\widehat b_{\kappa} \right) = \sum_{\ell = 1}^M  \sum_{m \in C_j}\frac{\gamma_{m,\ell}^2 \sigma_\ell^2 n^{-\alpha_\ell} |m|^{1 -2H_\ell}|g_{m,\ell}|^2 |\Psi^\kappa_{m}|^2}{\left(  \sum_{\ell=1}^M \gamma_{m,\ell}|g_{m,\ell}|^2\right)^2}. \label{eq:varBetaK-simple}
\end{equation}
Using the Cauchy Schwarz-inequality we have,
\begin{align*}
    \left( \sum_{\ell=1}^M  \gamma_{m,\ell} |g_{m,\ell}|^2 \right)^2 & \le \left( \sum_{\ell=1}^M  \gamma_{m,\ell}^2 \sigma_\ell^2 n^{-\alpha_\ell}|m|^{1-2H_{\ell}}|g_{m,\ell}|^2\right) \left( \sum_{\ell=1}^M \sigma_\ell^{-2}n^{\alpha_\ell} |m|^{2H_{\ell}-1}|g_{m,\ell}|^2\right)
\end{align*}
with equality only if $\gamma_{m,\ell} = \gamma_{m,\ell}^* \coloneqq  n^{\alpha_\ell}\sigma_\ell^{-2}|m|^{2H_\ell-1}$. Use these choice of optimal weights, $\gamma_{m,\ell}^*$, starting with the case of regular-smooth convolution,
\begin{align*}
    \Var \left(\widehat b_{\kappa} \right) &= \sum_{m \in C_j}|\Psi^\kappa_{m}|^2\sum_{\ell = 1}^M\frac{\gamma_{m,\ell}^2\sigma_\ell^2 n^{-\alpha_\ell} |m|^{1-2H_\ell}|g_{m,\ell}|^2}{\left(  \sum_{\ell=1}^M \gamma_{m,\ell}|g_{m,\ell}|^2\right)^2}\nonumber\\
        &= \sum_{m \in C_j}  |\Psi^\kappa_m |^2\left( \sum_{\ell=1}^M \sigma_\ell^{-2}n^{\alpha_\ell} |m|^{2H_{\ell}-1}|g_{m,\ell}|^2\right)^{-1}\nonumber\\
        & \le C \int_\mathbb{R}   |\Psi(x) |^2\, dx\, \left( \sum_{\ell = 1}^M n^{\alpha_\ell}\inf_{x \in C_j}|x|^{2H_{\ell}-1} \inf_{y \in C_j }|g_{y,\ell}|^2\right)^{-1} \nonumber\\
        & = \mathcal O \left( \left( \sum_{\ell = 1}^M n^{\alpha_\ell}2^{j(1-\alpha_\ell-2\nu_\ell)}\right)^{-1}\right)\nonumber\\
        & = \mathcal O \left( \min_{1\le \ell \le M}n^{-\alpha_\ell} 2^{-j(1-\alpha_\ell-2\nu_\ell)}\right).
\end{align*}
Consider the case of box-car convolution. In particular, for $x \in \mathbb{R}$ define the distance to the nearest integer, $\norm{x} \coloneqq \inf \left\{ |x - r|: r \in \mathbb{Z} \right\}$. Then bounds can be given on the box-car Fourier coefficients with,
\[
	\frac{2\norm{m c_\ell}}{\left|\pi mc_\ell \right|}  \le |g_{m, \ell}| \le \frac{\norm{m c_\ell}}{\left| mc_\ell \right|},
\]
(see for example, p.298 of \cite{DeCanditiis-Pensky-2006}). Using this bound with  the same optimal weights $\gamma_{m,\ell}^*$ and the bound $|\Psi_m^\kappa| \le 2^{-j}$ with \eqref{eq:varBetaK-simple},
\begin{align*}
    \Var \left(\widehat b_{\kappa} \right) &= \sum_{m \in C_j}  |\Psi^\kappa_m |^2\left( \sum_{\ell=1}^M \sigma_\ell^{-2}n^{\alpha_\ell} |m|^{2H_{\ell}-1}|g_{m,\ell}|^2\right)^{-1}\nonumber\\
        & \le \tfrac{2}{\pi}2^{-j} n^{-\alpha_*} \sum_{m \in C_j}  m^2 \left( \sum_{\ell=1}^M c_\ell^{-2}\sigma_\ell^{-2} |m|^{2H_{\ell}-1}\norm{mc_\ell}^2\right)^{-1}\nonumber\\
	& = \mathcal O \left(  2^{j(\alpha^*-2)} n^{-\alpha_*} \sum_{m \in C_j}  m^2 \left( \sum_{\ell=1}^M \norm{mc_\ell}^2\right)^{-1} \right)\nonumber\\
        & = \mathcal O \left( n^{-\alpha_*} j2^{j(1+\alpha^*+1/M)}\right).
\end{align*}
The last bound follows from a result in the proof of Lemma 4 in \cite{DeCanditiis-Pensky-2006} where,
\[
    \sum_{m \in C_j}  m^2 \left(  \sum_{\ell=1}^M  \norm{mc_\ell}^2\right)^{-1} = \mathcal O \left( j2^{j(3+1/M)}\right).
\]
Consider the final case of the super smooth convolution. Using similar arguments it can be shown,
\begin{align}
  \Var \left(\widehat a_{\kappa} \right) & = \sum_{m \in C_j}  |\Phi^\kappa_m |^2\left( \sum_{\ell=1}^M \sigma_\ell^{-2}n^{\alpha_\ell} |m|^{2H_{\ell}-1}|g_{m,\ell}|^2\right)^{-1}\nonumber\\
        & \le C\sum_{m \in C_j}  |\Phi^\kappa_m |^2\, \left( \sum_{\ell = 1}^M n^{\alpha_\ell}\inf_{x \in C_j}|x|^{2H_{\ell}-1} \inf_{y \in C_j }|g_{y,\ell}|^2\right)^{-1} \nonumber\\
        & = \mathcal O \left( \left( \sum_{\ell = 1}^M n^{\alpha_\ell}2^{j(1-\alpha_\ell)} \inf_{y \in C_j }|y|^{-2\gamma_\ell }e^{-2\theta_\ell |y|^{\beta_\ell}}\right)^{-1}\right)\nonumber\\
        & = \mathcal O \left( \min_{1\le \ell \le M}n^{-\alpha_\ell} 2^{-j(1-\alpha_\ell - 2 \nu_{\ell})} e^{2\theta_\ell 2^{j\beta_\ell}}\right).\label{eq:varbound-super-smooth}
\end{align}
\subsection{The maxiset theorem}
\label{sec:thm-maxi}

For completeness, we give the statement of the following theorem that is borrowed from Theorem 6.1 in \cite{Kerkyacharian-Picard-2000}. We also refer to Section \ref{sec:tem} below for the definition of the Temlyakov property. First, we introduce some notation: $\mu$ will denote the measure such that for $j\in \mathbb{N},\; k\in \mathbb{N}$ and $0 < q < p$,
\begin{equation*}
     \mu\{(j,k)\}=\|\tau_j\psi_{j,k}\|_p^p=
    \tau_j^p2^{j(\frac{p}{2}-1)}\|\psi\|_p^p, 
\end{equation*}
\begin{equation*}
     l_{q,\infty}(\mu)=\left\{ f \in L^p,\; \sup_{\lambda>
    0}\lambda^q \mu\{(j,k) : \;
    |b_{j,k}|>\tau_j\lambda\}<\infty\right\}. 
\end{equation*}

\begin{theorem}
    \label{maxlp} Let $ p>1$, $0<q<p$, $ \{\psi_{j,k}, j\ge -1,\; k=0,1,\ldots,2^j\} $ be a periodised wavelet basis of $L^2(T)$, $T=[0,1]$, and $\tau_{j}$ be a positive sequence such that the heteroscedastic basis $\tau_{j}\psi_{j,k}$ satisfies the Temlyakov property. Suppose that $\Lambda_n $ is a set of pairs $(j,k)$ and that $c_n$ is a deterministic sequence tending to zero with
    \begin{equation}\label{eq:condlam}
        \sup_n\,\mu\{\Lambda_n\}\,c_n^{p}<\infty.
    \end{equation}
    If, for any $n$ and any pair  $\kappa=(j,k) \in \Lambda_n$, we have
    \begin{eqnarray}
        \label{eq:cond.1} \mathbb{E} |\widehat b_{\kappa} -  b_{\kappa} |^{2p}
        &\leq& C \,(\tau_{j}\,c_n)^{2p},
        \\
        \label{eq:cond.2} \PP \Big( |\widehat{b_{\kappa}}- b_{\kappa} |
        \geq \eta\, \tau_{j}\, c_n/2 \Big) &\leq& C \,( c_n^{2p}\wedge
        c_n^{4}),
    \end{eqnarray}
    for some positive constants $\eta$ and $C$, then, the wavelet based estimator
    \begin{equation}
        \widehat{f}_n(t)=\sum_{\kappa\in \Lambda_n}
        \,\widehat b_{\kappa}\1{\{|\widehat b_{\kappa}|\geq\,\eta\,\tau_{j}\,c_n\}}\psi_{\kappa}(t), \quad t \in T, \label{new}
    \end{equation}  is such that, for all positive integers $n$,
    \[
    \mathbb{E}  \|\widehat{ f_n} -f \|_{p}^p \leq C\, c_n^{p-q},
    \]
    if and only if
    \begin{equation}
        f(\cdot) \in l_{q,\infty}(\mu), \label{eq:maxi1}
    \end{equation}
    and \begin{equation} \sup_nc_n^{q-p} \| f-\sum_{\kappa\in\Lambda_n}
    {b_{\kappa}}\psi_{\kappa}\|_p^p<\infty. \label{eq:maxi2}
\end{equation}
\end{theorem}

This theorem identifies the `maxiset' of a general wavelet thresholding estimator of the form \eqref{new}. This is done  by using conditions \eqref{eq:maxi1} and \eqref{eq:maxi2} for an appropriate  choice of $q$. In the proofs of Theorems \ref{thm:main-1} and \ref{thm:main-2}, we will choose $q$ according to the {\em dense} or the {\em sparse} regions as follows
\begin{equation}
    \label{eq:q.dense} q=q_d:=\frac{(2\nu_*+1)p}{2s+2\nu_*+1}, \quad
    \hbox{if}\quad s\geq \frac{2\nu_*+1}{2}\Big(\frac{p}{\pi_0}-1\Big )
\end{equation}
\begin{equation}
    \label{eq:q.sparse} q=q_s:=\frac{(2\nu_*+1)
    p/2-1}{s-1/\pi_0+(2\nu_*+1)/2},\quad \quad \hbox{if}\quad s \leq
    \frac{2\nu_*+1}{2}\Big(\frac{p}{\pi_0}-1\Big ).
\end{equation}
We first verify \eqref{eq:condlam}. Consider first the case of regular-smooth convolutions. Using \eqref{eq:j1}, simple algebra shows that
\begin{align*}
\mu(\{\Lambda_n\}) &= \sum_{j\le j_1}\sum_{k=0}^{2^j-1}\mu\{(j,k)\}=\sum_{j\le j_1}2^j\mu\{(j,k)\} \nonumber \\
&= \mathcal O (1)\sum_{j\le j_1}2^j2^{j(p/2-1)}\tau_j^p=\mathcal O (1)\sum_{j\le
j_1}2^{jp(1/2+\nu_*)} \nonumber \\
&=\mathcal O (2^{j_1p(1/2+\nu_*)})=\mathcal O (c_n^{-p}), 
\end{align*}
where $c_n$ is given by \eqref{eq:cn}, since it is easily seen in this case that $\tau_j^2=\mathcal O (2^{2j\nu_*})$ with $\nu_*$ given by \eqref{eq:smooth.optimal.nu} (compare also with p. 306 of \cite{DeCanditiis-Pensky-2006}). A similar bound can be shown for the box-car case with $\nu_*$ replaced with $\widetilde \nu_*$ given by \eqref{eq:box.car.optimal.nu}.

We now verify \eqref{eq:cond.1} and \eqref{eq:cond.2}. Since the random variables $\widehat b_{\kappa}-b_{\kappa}$ follow a Gaussian distribution, the higher moment bounds
\eqref{eq:cond.1} follows from the variance inequality. Similarly, denoting $Z$ to be a standard Gaussian distributed random variable,
\begin{align*}
    \PP\left(|\widehat b_{\kappa}-b_{\kappa}|>\zeta\tau_j c_n/2 \right) &= 2\PP \left( Z \ge \frac{\zeta\sqrt{\log n}}{2} \right)\nonumber\\
        & \le \frac{4n^{-\zeta^2/8}}{\zeta \sqrt{\log n}}\nonumber \\
        &= O\big( c_n^4\wedge c_n^{2p}\big).
\end{align*}
as long as $\zeta> 2\sqrt{(p\vee 2)2\xi}$. This proves \eqref{eq:cond.2}.

\subsection{Temlyakov property}
\label{sec:tem}

As seen in Appendix A in \cite{Johnstone-et-al-2004}, the basis $\{\tau_j \psi_{j,k}(\cdot)\}$ satisfies the Temlyakov property as soon as
\begin{equation}
\label{tem:1-fan} \sum_{j \in \Lambda_n} 2^j\, \tau_j^2\leq
C\sup_{j \in \Lambda_n} \big ( 2^j \tau_j^2\big )
\end{equation}
and
\begin{equation}
\label{tem:2-fan} \sum_{j \in \Lambda_n} 2^{jp/2}\, \tau_j^p\leq C\,
\sup_{j \in \Lambda_n} \big ( 2^{jp/2} \tau_j^p\big ),\quad 1\leq p<2.
\end{equation}
Recall that $\tau_j^2=\mathcal O (2^{2\nu_*j})$ (regular-smooth convolutions) and $\tau_j^2=O\big(j2^{2\widetilde \nu_*j}\big)$ (box-car convolutions) with $\nu_*$ and $\widetilde \nu_*$ given by \eqref{eq:smooth.optimal.nu} and \eqref{eq:box.car.optimal.nu}. Hence, \eqref{tem:1-fan} and \eqref{tem:2-fan} are verified by direct calculations.

\subsection{Besov embedding and maxiset conditions}
\label{sec:imbII}

We recall that
\begin{equation}
\label{eq:Inc1} \Besov{\pi_0}{r}{s} \subseteq \Besov{p}{r}{s''},\quad
\hbox{if}\quad  \pi_0\geq p, \quad s\geq s'' .
\end{equation}
\begin{equation}
\label{eq:Inc2} \Besov{\pi_0}{r}{s} \subseteq \Besov{p}{r}{s''},\quad
\hbox{if} \quad \pi_0 < p, \quad s-1/\pi_0 = s''-1/p.
\end{equation}
For both {\em dense} (\ref{eq:q.dense}) and {\em sparse}
(\ref{eq:q.sparse}) regions, we look for a Besov scale $\delta$ such
that 
 $\Besov{\pi_0}{r}{\delta} \subseteq l_{q,\infty}.$ 
As usual, we note that it is easier to work with
\[
l_q(\mu)=\left\{f(\cdot) \in L^p(T):~f=\sum_{j,k}b_{j,k}\psi_{j,k} \;\;\text{such that}\;\; \sum_{j,k\in
A_j}\frac{|\beta_{jk}|^q}{\tau_j^q}\left\|\tau_j\psi_{j,k}\right\|_p^p<\infty\right\},
\]
where $A_j$ is a set of cardinality proportional to $2^j$. Using
(\ref{eq:sigma.j}) and the fact that
\[
\left\|\tau_j\psi_{j,k}\right\|_p^p=\tau_j^p\,
2^{j(p/2-1)}=2^{j((2\nu_*+1) p/2-1)},
\]
we see that $f(\cdot)\in l_q(\mu)$ if,
\[
\sum_{j\geq 0} 2^{j ( (2\nu_*+1) p-2 -2\nu_*q)
/2}\sum_{k=0}^{2^j-1} |b_{j,k}|^q= \sum_{j\geq 0} 2^{jq \Big [
\frac{(2\nu_*+1)(p-q) }{2q}+\frac{1}{2}-\frac{1}{q}\Big ]
}\sum_{k=0}^{2^j-1} |b_{j,k}|^q <+\infty.
\]
From \eqref{eq:fBesov}, the latter condition holds when
\begin{equation}
\label{eq:delta.s} f(\cdot) \in \Besov{q}{q}{\delta}(T) \quad
\text{for} \quad \delta=\frac{(2\nu_*+1)}{2}\bigg
(\frac{p}{q}-1\bigg ).
\end{equation}
Now, depending on whether we are in the {\em dense} \eqref{eq:q.dense} or {\em sparse}  \eqref{eq:q.sparse} regions, we
look for $s$ and $\pi$ such that
\begin{equation}
\label{eq:bsvEMB} \Besov{\pi_0}{r}{s}  \subseteq
\Besov{q}{q}{\delta}.
\end{equation}
This embedding can be found by exploiting the known monotonicity of Besov balls, namely for $0 < r \leq q$, $\Besov{\pi_0}{r}{s} \subseteq \Besov{\pi_0}{q}{s}$, along with \eqref{eq:Inc1} or \eqref{eq:Inc2}.

\noindent{\bf The dense region.} By definition \eqref{eq:q.dense} of $q=q_d$, we  have $s \ge (\nu_*+1/2)(p/\pi_0 - 1)$. Eliminate $p$ by substituting $p = q_d (2s + 2\nu_* +1)/(2\nu_*+1)$ yields $\pi_0\geq q_d$. Hence, \eqref{eq:bsvEMB} follows from \eqref{eq:Inc1} as long as $s\geq \delta=\frac{(2\nu_*+1)}{2}(\frac{p}{q}-1 )$, which is always true in this {\em dense} region since $\delta = s > 0$ when $q = q_d$.

\vskip2mm \noindent{\bf The sparse region.} Take $q=q_s$ and $\delta=\frac{(2\nu_*+1)}{2}\left(\frac{p}{q_s}-1\right)=(2\nu_*+1)\frac{sp-p/\pi_0+1}{(2\nu_*+1) p-2}$. We consider two cases. If $\pi_0>q_s$, we use the embedding \eqref{eq:Inc1}. We have to check that $s> \delta = (2\nu_*+1)\frac{sp-p/\pi_0+1}{(2\nu_*+1) p-2}$ which is equivalent to $s<\frac{(2\nu_*+1)}{2}\left(\frac{p}{\pi_0}-1\right)$, which is true in the {\em sparse} region. Note that we require $\delta>0$ which implies either $(i)$ $p > 2/(2\nu_*+1)$ and $s > 1/\pi_0 - 1/p$ or $(ii)$  $p < 2/(2\nu_*+1)$ and $s < 1/\pi_0 - 1/p$. The $(ii)$ scenario is impossible since $p< 2/(2\nu_*+1)$ and $s < 1/\pi_0 - 1/p$ is a contradiction of $s \ge 1/\pi_0 - \nu_* - 1/2$. On the other hand, by definition, when in the sparse phase, $1/\pi_0 - \nu_* - 1/2 < (\nu_*+1/2)(p/\pi_0 - 1)$ which implies $p > 2/(2\nu_*+1)$ and consequently verifies that $s > 1/\pi_0 - 1/p$ since $s > 1/\pi_0 - \nu_* - 1/2$. Thus we established \eqref{eq:bsvEMB} for $q_s<\pi_0<q_d$. By definition \eqref{eq:q.sparse} of $q=q_s$, if $\pi_0 \leq q_s$, the corresponding $\delta$ fulfils $s-1/\pi_0=s'-1/q$. In this case, \eqref{eq:Inc2} and \eqref{eq:delta.s} ensure that
\[
 \Besov{\pi_0}{r}{s}  \subseteq  \Besov{q}{q}{s'}\equiv l_q(\mu),
\]
as had to be proved.

To apply \autoref{maxlp}, \eqref{eq:maxi2} needs to be verified. Therefore we need to find a $\delta > 0$ such that for any $f \in \Besov{p}{r}{\delta}$, \eqref{eq:maxi2} is satisfied.
\begin{align*}
    c_n^{q-p} \norm{f - \sum_{j,k }b_{j,k}\Psi_{j,k}}_p^p = c_n^{q-p}2^{-j_1 \delta p} \norm{f}_{\Besov{p}{r}{\delta}} = c_n^{q-p+ 2\delta p/(2\nu_* + 1)} \norm{f}_{\Besov{p}{r}{\delta}}.
\end{align*}
The above is bounded uniformly in $n$ if we choose $\delta = 1/2(2\nu_* + 1)(1 - q/p)$. Now we need to find $s,\pi_0$ such that $\Besov{\pi_0}{r}{s}\subseteq \Besov{p}{r}{\delta}$.

Consider the first case $\pi_0 \ge p$. This case cannot occur in the sparse phase due to \eqref{eq:rates-sparse-smooth} and \eqref{eq:rates-sparse-boxcar} with the assumption that $s$ is positive. In the dense phase, use embedding \eqref{eq:Inc1} with $\gamma = \delta$ and $q = q_d$. Therefore, \eqref{eq:Inc1} holds if $s \ge 1/2(2\nu + \alpha)(1 - q_d/p)$. This implies,
\begin{align*}
    s &\ge 1/2(2\nu_* + 1)(1 - q_d/p)\\
    &= \frac{(2\nu_* + 1)s}{2s + 2\nu_* + 1},
\end{align*}
which always holds under the assumption that $s > 0$.

Now consider the dense case when $\pi_0 < p$. In this scenario use embedding \eqref{eq:Inc2} by defining $s - 1/\pi_0 = s'' - 1/p$ which ensures $\Besov{\pi_0}{r}{s} \subseteq \Besov{p}{r}{s''}$. Then complete the embedding using \eqref{eq:Inc1} (namely, $\Besov{p}{r}{s''} \subseteq \Besov{p}{r}{\delta}$) which requires $s'' \ge \delta$ with $q = q_d$ or equivalently after rearrangement, $2ss'' + (2\nu_* + 1)(1/p - 1/\pi_0) \ge 0.$ The left hand side is greater than $(s - 1/\pi_0)(p/\pi_0 - 1)(2\nu_* + 1) \ge 0$ when $s \ge (\nu _*+ 1/2)(p/\pi_0 -1)$ (which is true in the dense phase).

The last case to consider is the sparse case when $\pi_0 < p$. Again introduce a new Besov scale $s''$ defined with, $s - 1/\pi = s'' - 1/p$ and apply a similar argument to above which requires that, $s'' \ge \delta$ with $q = q_s$. This is satisfied if $s > 1/\pi_0$, which always holds.

\subsection{Proofs of \autoref{thm:main-1} and \autoref{thm:main-2}}

The proofs of Theorems \ref{thm:main-1} and \ref{thm:main-2} are a direct application of \autoref{maxlp} with $j_1$, $\zeta$, $\tau_j$, and $c_n$ of Section \ref{sec:2}. Combining results of Sections \ref{sec:tem} and \ref{sec:imbII}, we see that all conditions of \autoref{maxlp} are satisfied. Using the embedding results of Section \ref{sec:imbII}, we derive the rate exponent $\gamma = \gamma_S$ or $\gamma = \gamma_B$ given by \eqref{eq:rates-dense-smooth} and \eqref{eq:rates-dense-boxcar} respectively for smooth and boxcar convolutions for any $f(\cdot) \in \Besov{\pi_0}{r}{s}$ using \eqref{eq:q.dense} for $q$ when $s\geq \frac{(2\nu_*+1)}{2}(p/\pi_0-1)$ and the rate exponent $\gamma = \gamma_S$ and $\gamma  = \gamma_B$  given by \eqref{eq:rates-sparse-smooth} and \eqref{eq:rates-sparse-boxcar} respectively for smooth and boxcar convolutions for any $f(\cdot) \in \Besov{\pi_0}{r}{s}$ using \ref{eq:q.sparse} for  $q$ when $1/\pi_0 \leq s < \frac{(2\nu_*+1)}{2}(p/\pi_0-1)$, with $\nu_*$ given either by \eqref{eq:smooth.optimal.nu} (regular-smooth convolutions) or \eqref{eq:box.car.optimal.nu} (box-car convolutions).

For the super-smooth scenario in \autoref{thm:main-1} we appeal to the same arguments used in the proof of \cite[Theorem 4.2][]{Petsa-Sapatinas-2009}. Consider the moment bound directly with the estimator \eqref{eq:super-smooth.estimator},
\begin{align}
  \label{eq:main-3-1} \mathbb{E} \norm{\widehat f_n - f}^p_p &\le 2^{p-1} \mathbb{E} \norm{\sum_{k = 0}^{2^{j_0} - 1} \left(\widehat a_{j_0,k} - a_{j_0,k} \right)\Phi_{j_0,k}(t) }^p_p\\
    &\qquad \qquad + 2^{p-1} \norm{\sum_{j = j_0}^{\infty} \sum_{k = 0}^{2^j-1} b_{j,k}\Psi_{j,k}(t) }_p^p\nonumber
\end{align}
The two terms in \eqref{eq:main-3-1} can be bounded separately with \eqref{eq:varbound-super-smooth} and the scale level \eqref{eq:j0.super-smooth},
\begin{align}
  \mathbb{E} \norm{\sum_{k = 0}^{2^{j_0} - 1} \left(\widehat a_{j_0,k} - a_{j_0,k} \right)\Phi_{j_0,k}(t) }^p_p &\le C 2^{j_0(p/2-1)} \sum_{k = 0}^{2^{j_0}-1} \,  \mathbb{E} |\widehat a_{j_0,k} - a_{j_0,k}|^p\nonumber\\
  &\le C  n^{-\alpha_{\ell_*}p/2}2^{j_0p/2(\alpha_{\ell_*}+ 2\nu_{\ell_*})}e^{a_{\ell_*}p2^{j_0\beta_{\ell_*}}}\nonumber\\
   &\le C  n^{-\epsilon p/2}(\log n)^{p/2(\alpha_{\ell_*}+ 2\nu_{\ell_*})}\nonumber\\
  &= o((\log n)^{-s^*p/\beta_{\ell_*}}).\label{eq:main-3-2}
\end{align}
For the next term use the property of Besov spaces,
\begin{align}
  \norm{\sum_{j = j_0}^{\infty} \sum_{k = 0}^{2^j-1} b_{j,k}\Psi_{j,k}(t) }_p^p &\le \left( \sum_{j = j_0}^\infty C 2^{-j(s + 1/p - 1/\min (\pi_0,p))}  \right)^p \nonumber\\
  &= \mathcal O((\log n)^{-s^*p/\beta_{\ell_*}}).\label{eq:main-3-3}
\end{align}
The result of \eqref{eq:rate-super-smooth} follows combining the results \eqref{eq:main-3-1}, \eqref{eq:main-3-2} and \eqref{eq:main-3-3}.

\section*{Acknowledgements}

The authors would like to thank the the Editor and two anonymous reviewers whose comments and suggestions lead to an improved version and presentation of the paper.

 \appendix
 \label{sec:appendix}
\section{Meyer Wavelet Proofs}
\begin{lemma}
\label{lem:dyadic}
Let $\omega \in \mathbb{Z}$ and $j \in \mathbb{Z}^+$, then the following identity holds for the sum of the dyadic rationals on the complex unit circle.
\[
  \sum_{k = 0}^{2^j-1} e^{2\pi i \omega k 2^{-j}} = 2^j \1{\left\{ \tfrac{\omega}{2^j} \in \mathbb{Z} \right\}} 
\]
\end{lemma}
\begin{proof}
The proof relies on the trigonometric components (real and imaginary parts) of the complex exponential. Namely,
\[
	e^{2\pi i \omega k 2^{-j}} = \cos \left( 2\pi \omega k 2^{-j} \right) + i \sin \left( 2\pi \omega k 2^{-j} \right).
\]
The case when $\omega = 0$ follows immediately due to the identities that $\cos (0) = 1$ and $\sin(0) = 0$.
Consider now the case when $\omega \neq 0$. Starting with the real part, partition the summation into halves with,
\begin{align*}
	\sum_{k = 0}^{2^j-1}\cos \left(2\pi \omega k 2^{-j} \right) &= \sum_{k = 0}^{2^{j-1}-1}\cos \left(2\pi \omega k 2^{-j} \right) + \sum_{k = 2^{j-1}}^{2^j-1}\cos \left(2\pi \omega k 2^{-j} \right)\\
	&= \sum_{k = 0}^{2^{j-1}-1}\cos \left(2\pi \omega k 2^{-j} \right) + \sum_{k = 0}^{2^j-2^{j-1}-1}\cos \left(2\pi \omega k 2^{-j} + \pi \omega \right)\\
	&= \left( 1+(-1)^\omega\right)\sum_{k = 0}^{2^{j-1}-1}\cos \left(2\pi \omega k 2^{-j} \right).
\end{align*}
If $\omega$ is odd, $\omega = 2s+1$ for some $s \in \mathbb{Z}$ then the above result is zero. Therefore consider $\omega = 2s$ for some $s \in \mathbb{Z}$ ($\omega$ is even). 
\begin{align*}
	\sum_{k = 0}^{2^j-1}\cos \left(2\pi \omega k 2^{-j} \right) &= \left( 1+(-1)^\omega\right)\sum_{k = 0}^{2^{j-1}-1}\cos \left(2\pi \omega k 2^{-j} \right)\\
	&=  2\sum_{k = 0}^{2^{j-1}-1}\cos \left(2\pi \omega k 2^{-j} \right)\\
	&=  2 \left\{  \sum_{k = 0}^{2^{j-2}-1}\cos \left(2\pi \omega k 2^{-j} \right) + \sum_{k = 2^{j-2}}^{2^{j-1}-1}\cos \left(2\pi \omega k 2^{-j} \right)\right\}\\
	&=  2 \left\{  \sum_{k = 0}^{2^{j-2}-1}\cos \left(2\pi \omega k 2^{-j} \right) + \sum_{k =0}^{2^{j-2}-1}\cos \left(2\pi \omega k 2^{-j}  + \tfrac{\pi \omega}{2}\right)\right\}\\
	&= 2 \left\{  \sum_{k = 0}^{2^{j-2}-1}\cos \left(2\pi \omega k 2^{-j} \right) + \sum_{k =0}^{2^{j-2}-1}\cos \left(2\pi \omega k 2^{-j}  + \pi s\right)\right\}\\
	&= 2 (1 + (-1)^s ) \sum_{k = 0}^{2^{j-2}-1}\cos \left(2\pi \omega k 2^{-j} \right).
\end{align*}
If $s$ is odd then the above result is zero. This process can be repeated until we reach the last possible result where $\omega = C2^j$ for some $C \in \mathbb{Z}$ and,
\[ 
	\sum_{k = 0}^{2^j-1}\cos \left(2\pi \omega k 2^{-j} \right) = 2^j \1{\left\{ k = 0\right\}} \cos \left(2\pi \omega k 2^{-j} \right) = 2^j.
\]
A similar proof applies for the imaginary part except the final step has,
\[
\sum_{k = 0}^{2^j-1}\sin \left(2\pi \omega k 2^{-j} \right) = 2^j \1{\left\{ k = 0\right\}} \sin \left(2\pi \omega k 2^{-j} \right) = 0.\qedhere
\]
\end{proof}

\begin{lemma}
\label{lem:identity-covariance}
Let $(\phi,\psi)$ be the Meyer wavelet basis. That is, the mother Meyer wavelet defined in the Fourier domain with,
\begin{equation}
  \psi_m = \int_\mathbb{R} e^{-2\pi i m x} \psi (x) \, dx = e^{i \pi m} \begin{cases} \sin\left( \frac{\pi}{2} \nu(3|m|-1)  \right) & \quad \text{for } \frac{1}{3} \le |m| \le \frac{2}{3};\\
  \cos\left( \frac{\pi}{2} \nu\left(\frac{3}{2}|m|-1 \right)  \right) & \quad \text{for } \frac{2}{3} \le |m| \le \frac{4}{3};\\
  0 & \qquad \text{otherwise},
  \end{cases}\label{eq:FourierWavelet}
\end{equation}
where $\nu(x)$ is a polynomial that controls the vanishing moment properties of the wavelet basis. In particular, the Meyer wavelet has the defining property that the polynomial satisfies,
\begin{equation}
  \nu(x) + \nu(1 - x) = 1. \label{eq:MeyerPoly}
\end{equation}
Then the matrix $\boldsymbol{M} = \left( M_{m,m'} \right)_{m,m'\in C_j}$ defined with entries 
\[
   M_{m,m'}  = \sum_{j' \in \mathbb{Z}}\1{\left\{ \frac{m-m'}{2^j} \in \mathbb{Z} \right\}}\psi_{m2^{-j'}} \overline{\psi_{m'2^{-j'}}}
\]
is the identity matrix.
\end{lemma}

\begin{proof}
Using the definition of \eqref{eq:FourierWavelet} and considering $a \in \mathbb{Z}$ such that $(m-m')2^{-j} = a$ we can write,
\[
	\1{\left\{ \frac{m-m'}{2^j} \in \mathbb{Z} \right\}}\psi_{m2^{-j}} \overline{\psi_{m'2^{-j}}} = e^{i a\pi} f(|m2^{-j}|)f(|m'2^{-j}|)
\]
where $f$ is defined by the piece wise trigonometric functions given in \eqref{eq:FourierWavelet}. Recall the support of the Meyer wavelet at scale $j$ in the Fourier domain is $C_j$ defined in \eqref{eq:Cj}. Define a partition of the domain at level $j$ with the domain at the surrounding scales, $j-1$ and $j+1$ with,
\begin{align*}
  C_j \cap C_{j-1} & = \left\{ a \in \mathbb{Z} :  \frac{1}{3} \le |a2^{-j}| \le \frac{2}{3} \right\} \eqqcolon C_{j-1}^{\sin} \\
  C_j \cap C_{j+1} &= \left\{ a \in \mathbb{Z} :  \frac{2}{3} \le |a2^{-j}| \le \frac{4}{3} \right\} \eqqcolon C_{j+1}^{\cos}. 
\end{align*}
The sets are named $C_{j-1}^{\sin}$ and $C_{j+1}^{\cos}$ respectively since they refer to those trigonometric parts of the Meyer wavelet at scale $j$ respectively (see \eqref{eq:FourierWavelet}) and it is the domain where the coefficients are in both $C_j \cap C_{j-1}$ and $C_j \cap C_{j+1}$ respectively. To ease the tedious nature of the forthcoming argument, consider a particular ordering of the two sets $C_{j+1}^{\cos}$ and $C_{j-1}^{\sin}$.
\begin{align*}
	C_{j-1}^{\sin} &= \left\{ -\lfloor \tfrac{2^{j+1}}{3} \rfloor, -\lfloor \tfrac{2^{j+1}}{3} \rfloor+1,\ldots, -\lceil \tfrac{2^j}{3} \rceil, \lceil \tfrac{2^j}{3} \rceil,\lceil \tfrac{2^j}{3} \rceil + 1, \ldots, \lfloor \tfrac{2^{j+1}}{3} \rfloor\right\}.\\
C_{j+1}^{\cos} &= \left\{ -\lfloor \tfrac{2^{j+2}}{3} \rfloor , -\lfloor \tfrac{2^{j+2}}{3} \rfloor+1,\ldots, -\lceil \tfrac{2^{j+1}}{3} \rceil,\lceil \tfrac{2^{j+1}}{3} \rceil,\lceil \tfrac{2^{j+1}}{3} \rceil+1, \ldots, \lfloor \tfrac{2^{j+2}}{3} \rfloor\right\}.
\end{align*}
 Further partition these sets into the positive and negative parts with $C_{j-1}^{\sin+} \coloneqq C_{j-1}^{\sin} \cap \mathbb{Z}^+$, $C_{j-1}^{\sin-}\coloneqq C_{j-1}^{\sin} \cap \mathbb{Z}^-$,$C_{j+1}^{\cos+}\coloneqq C_{j-1}^{\cos} \cap \mathbb{Z}^+$ and $C_{j+1}^{\cos-}\coloneqq C_{j-1}^{\cos} \cap \mathbb{Z}^-$ where $\mathbb{Z}^+$ and $\mathbb{Z}^-$ are the positive and negative integers respectively. 

Write the matrix $\boldsymbol{M}$ in the following way,
\[ 
	\boldsymbol{M} = 
\begin{blockarray}{cc|c|c|c}
    & C_{j+1}^{\cos-} & C_{j-1}^{\sin-} & C_{j-1}^{\sin+}&  C_{j+1}^{\cos+} \\
    \begin{block}{c(c|c|c|c)}
    C_{j+1}^{\cos-} & \boldsymbol{E_1} & \boldsymbol{0} & \boldsymbol{0} &  \boldsymbol{E_2}\\
\cline{1-5}
    C_{j-1}^{\sin-} & \boldsymbol{0} & \boldsymbol{R_1} & \boldsymbol{R_2} &  \boldsymbol{0}\\
\cline{1-5}
    C_{j-1}^{\sin+} & \boldsymbol{0} & \boldsymbol{R_3} & \boldsymbol{R_4} &  \boldsymbol{0}\\
\cline{1-5}
    C_{j+1}^{\cos+} &  \boldsymbol{E_3} & \boldsymbol{0} & \boldsymbol{0} &  \boldsymbol{E_4}\\
    \end{block}
  \end{blockarray}
\]
where the outer sets denote the values of $m,m'$ inside the $\boldsymbol{M}$ matrix. With a slight abuse of notation we will refer to the elements of $\boldsymbol{M}$ using $m,m'\in C_j$. For example, the first element $M_{1,1}$(the top left matrix entry of $\boldsymbol{E_1}$) has $m = m' = -\lceil \tfrac{2^{j+1}}{3} \rceil$, the first element of $C_{j+1}^{\cos-}$.

The matrix $\boldsymbol{M}$ is composed of block matrix components where $\boldsymbol{0}$ denotes a matrix of zeros of appropriate size implied by the cardinalities of $C_{j+1}^{\cos-}$, $C_{j-1}^{\sin-}$, $C_{j-1}^{\sin+}$ and $C_{j+1}^{\cos+}$.  The zero matrices follow since a value $m \in C_j$ cannot be in both $C_{j-1}$ and $C_{j+1}$ since $C_{j-1} \cap C_{j+1} = \emptyset$.  The overall result follows by showing that the other block matrices simplify to the following: $\boldsymbol{E_1} =\boldsymbol{R_1}=\boldsymbol{R_4}=\boldsymbol{E_4} = \boldsymbol{I}$ and $\boldsymbol{E_2} =\boldsymbol{R_2}=\boldsymbol{R_3}=\boldsymbol{E_3} = \boldsymbol{0}$ where $\boldsymbol{I}$ is the identity matrix of appropriate size.  To show these results for each case, one needs to first consider the values of $m,m' \in C_j$ such that $m-m'$ is a factor of $2^k$ for $k \in \left\{ j,j-1,j+1 \right\}$ and then compute the sum $\sum_{k = j-1}^{j+1} \psi_{m2^{-k}} \overline{\psi_{m2^{-k}}}$ for these values. 

Before proceeding, some notation is defined. For $x \in \mathbb{R}$, let $\left\{ x\right\}$ denote the fractional part of $x$. Then we have,
\[
	\lfloor x \rfloor = x - \left\{ x \right\} \qquad \text{and} \qquad \lceil x \rceil = x +1 - \left\{ x \right\}.
\]

{\sl Case $\boldsymbol{R_1}$ and $\boldsymbol{R_4}$}\\
We will consider here only the case for $\boldsymbol{R_4}$, the case for $\boldsymbol{R_1}$ follows by symmetry. In this context, $m,m' \in C_{j-1}^{\sin +}$ where it is possible that $m = m'$ is a solution to $m-m' = s2^j$ with $s = 0.$ This is in fact the only solution since the cardinality of $C_{j-1}^{\sin +} < 2^{j-1}$. Indeed,
\begin{align*}
	 |C_{j-1}^{\sin+}|&= \lfloor \tfrac{2^{j+1}}{3} \rfloor-\lceil \tfrac{2^j}{3} \rceil +1\\
		&= \tfrac{2^{j+1}}{3} - \left\{ \tfrac{2^{j+1}}{3} \right\}- \tfrac{2^j}{3} - 1 +\left\{ \tfrac{2^j}{3}\right\}+1\\
		&= \frac{2^j-(-1)^j}{3} < 2^{j-1}.
\end{align*}
Therefore the only value of $s\in\mathbb{Z}$ such that $m-m' = s2^{j}$ or $m-m' = s2^{j-1}$ is $s = 0$ ($m = m'$). This scenario occurs along the diagonal of $\boldsymbol{R_4}$, therefore the off-diagonal elements are zero. Computing the diagonal elements, we have $m,m' \in C_{j-1}^{\sin+} \implies 2m,2m' \in C_{j+1}^{\cos+}$ and $\tfrac{m}{2},\tfrac{m'}{2} \notin C_j$. Therefore only the scales $j-1$ and $j$ are used in the summation. Consider these diagonal elements of $\boldsymbol{R_4}$ with,
\begin{align*}
	R_4 &= \1{\left\{ m,m' \in C_{j-1}^{\sin+} : m=m'\right\}}\left\{ \psi_{m2^{-j}} \overline{\psi_{m'2^{-j}}} + \psi_{m2^{-j+1}} \overline{\psi_{m'2^{-j+1}}}  \right\}\\
		&=  \1{\left\{ m,m' \in C_{j-1}^{\sin+} : m=m'\right\}}\psi_{m2^{-j}} \overline{\psi_{m'2^{-j}}} \\
		&\qquad \qquad + \1{\left\{ 2m,2m' \in C_{j+1}^{\cos+} : m=m'\right\}}\psi_{2m2^{-j}} \overline{\psi_{2m'2^{-j}}}   \\
		&= \1{\left\{ m = m'\right\}}\left( \sin^2\left( \frac{\pi}{2} \nu(3|m2^{-j}|-1)  \right) + \cos^2\left( \frac{\pi}{2} \nu(\tfrac{3}{2}|m2^{-j+1}|-1)\right)  \right) \\
		&= 1
\end{align*}
since $\sin^2\theta + \cos^2 \theta = 1$ for all $\theta \in \mathbb{R}$.

{\sl Case $\boldsymbol{R_2}$ and $\boldsymbol{R_3}$}\\
Similarly, we will consider here only the case for $\boldsymbol{R_3}$, the case for $\boldsymbol{R_2}$ follows by symmetry. In this context, $m\in C_{j-1}^{\sin +}$ and $m'\in C_{j-1}^{\sin -}$. Consider the values $m-m'$ along the main diagonal of $\boldsymbol{R_3}$ which are identical since the values $m\in C_{j-1}^{\sin +}$ and $m'\in C_{j-1}^{\sin -}$ are consecutive. The first diagonal element is when $m = \lfloor \tfrac{2^{j+1}}{3} \rfloor$ and $m' = -\lceil \tfrac{2^j}{3} \rceil$ yielding,
\begin{align*}
	 m-m' &= \lfloor \tfrac{2^{j+1}}{3} \rfloor+\lceil \tfrac{2^j}{3} \rceil \\
		&= \tfrac{2^{j+1}}{3} - \left\{\tfrac{2^{j+1}}{3} \right\}+\tfrac{2^j}{3} + 1 -\left\{ \tfrac{2^j}{3}\right\}\\
		&= 2^j - \left\{\tfrac{2^{j+1}}{3} \right\} + 1 -\left\{ \tfrac{2^j}{3}\right\}\\
		&= 2^j  - \left\{ 2^j-\tfrac{2^j}{3} \right\} - \left\{ \tfrac{2^j}{3} \right\}+1\\
		&=2^j -1 +1 = 2^j.
\end{align*}
Similarly the maximum and minimum distances are, $\max_{m \in C_{j-1}^{\sin+},m' \in C_{j-1}^{\sin-}}(m-m')= 2\lfloor \tfrac{2^{j+1}}{3} \rfloor$ and $\min_{m \in C_{j-1}^{\sin+},m' \in C_{j-1}^{\sin-}}(m-m')= 2\lceil \tfrac{2^j}{3} \rceil$. Therefore the range of possible distances between $m$ and $m'$ are of length $2 \left(\lfloor \tfrac{2^{j+1}}{3} \rfloor -  \lceil \tfrac{2^j}{3} \rceil\right) =\frac{2^{j+1}-2(-1)^j}{3} - 2 < 2^j$. Therefore $m-m' = s2^j$ for some $s \in \mathbb{Z}$ only on the diagonal and $s = 1$ in this case ($m = m' + 2^j$). Thus again,  $\boldsymbol{R_3}$ is a diagonal matrix. Computing these values we have, $m \in C_{j-1}^{\sin+} \implies 2m \in C_{j+1}^{\cos+}$ and $\tfrac{m}{2} \notin C_j$. Similar cases apply to $m'$. Therefore only the scales $j-1$ and $j$ are used in the summation. Consider these diagonal elements of $\boldsymbol{R_3}$,
\begin{align*}
	R_3 &= \1{\left\{ m \in C_{j-1}^{\sin+} ,m' \in C_{j-1}^{\sin-} : m=m'+2^j\right\}} \left\{  \psi_{m2^{-j}} \overline{\psi_{m'2^{-j}}} + \psi_{m2^{-j+1}} \overline{\psi_{m'2^{-j+1}}}  \right\}\\
		&=  \1{\left\{ m=m'+2^j\right\}} \psi_{m2^{-j}} \overline{\psi_{m'2^{-j}}} \1{\left\{ m \in C_{j-1}^{\sin+} ,m' \in C_{j-1}^{\sin-}\right\}}\\
		&\qquad  + \1{\left\{ m=m'+2^j\right\}} \psi_{2m2^{-j}} \overline{\psi_{2m'2^{-j}}}  \1{\left\{ 2m \in C_{j+1}^{\cos+} ,2m' \in C_{j+1}^{\cos-}\right\}} \\
		&= \1{\left\{ m=m' + 2^j\right\}} e^{i \pi(m-m')2^{-j}}\sin\left( \frac{\pi}{2} \nu(3|m2^{-j}|-1)  \right)\sin\left( \frac{\pi}{2} \nu(3|m'2^{-j}|-1) \right) \\
		&\qquad + \1{\left\{ m=m'+ 2^j\right\}} e^{i \pi(m-m')2^{-{j-1}}}\cos\left( \frac{\pi}{2} \nu(\tfrac{3}{2}|m2^{-j+1}|-1)\right)  \cos\left( \frac{\pi}{2} \nu(\tfrac{3}{2}|m'2^{-j+1}|-1)\right)\\
		&= -\1{\left\{ m=m' + 2^j\right\}} \sin\left( \frac{\pi}{2} \nu(3|m2^{-j}|-1)  \right)\sin\left( \frac{\pi}{2} \nu(3|m'2^{-j}|-1) \right) \\
		&\qquad \qquad +\1{\left\{ m=m'+ 2^j\right\}}\cos\left( \frac{\pi}{2} \nu(3|m2^{-j}|-1)\right)  \cos\left( \frac{\pi}{2} \nu(3|m'2^{-j}|-1)\right)\\
		&=\1{\left\{ m=m'+ 2^j\right\}} \cos\left( \frac{\pi}{2} \nu(3|m2^{-j}|-1)+ \frac{\pi}{2} \nu(3|m'2^{-j}|-1)\right).
\end{align*}
Exploit now the fact that $\1{\left\{  m \in C_{j-1}^{\sin+} ,m' \in C_{j-1}^{\sin-} \right\}}$ implying $m > 0$ and $m' < 0$  along with the defining property of the Meyer wavelet in \eqref{eq:MeyerPoly} with the specific choice $x = 3m2^{-j}-1$,
\begin{align*}
	R_3 &= \1{\left\{ m=m'+ 2^j\right\}} \cos\left( \frac{\pi}{2} \nu(3|m2^{-j}|-1)+ \frac{\pi}{2} \nu(3|m'2^{-j}|-1)\right)\\
		&= \1{\left\{ -m'2^{-j}=1-m2^{-j}\right\}} \cos\left( \frac{\pi}{2} \nu(3m2^{-j}-1)+ \frac{\pi}{2} \nu(-3m'2^{-j}-1)\right)\\
		&= \cos\left( \frac{\pi}{2} \nu(3m2^{-j}-1)+ \frac{\pi}{2} \nu(3-3m2^{-j}-1)\right)\\
		&= \cos \left( \frac{\pi}{2}\right) = 0.
\end{align*}
Therefore $\boldsymbol{R_3} = \boldsymbol{0}$.

{\sl Case $\boldsymbol{E_1}$ and $\boldsymbol{E_4}$}\\
Similarly, we will consider here only the case for $\boldsymbol{E_4}$, the case for $\boldsymbol{E_1}$ follows by symmetry. In this context, $m,m'\in C_{j+1}^{\cos +}$ and we apply a similar argument used in the cases for $\boldsymbol{R_1}$ and $\boldsymbol{R_4}$.  Again, $m = m'$ is the only solution to $m-m' = s2^j$ with $s = 0\in\mathbb{Z}.$ Indeed,
\begin{align*}
	 |C_{j+1}^{\cos+}|&= \lfloor \tfrac{2^{j+2}}{3} \rfloor-\lceil \tfrac{2^{j+1}}{3} \rceil +1\\
		&= \tfrac{2^{j+2}}{3} - \left\{ \tfrac{2^{j+2}}{3} \right\}- \tfrac{2^{j+1}}{3} - 1 +\left\{ \tfrac{2^{j+1}}{3}\right\}+1\\
		&= \frac{2^{j+1}-(-1)^{j+1}}{3} < 2^{j}.
\end{align*}
Therefore the only value of $s\in\mathbb{Z}$ such that $m-m' = s2^{j}$ or $m-m' = s2^{j+1}$ is $s = 0$ ($m = m'$). This scenario occurs along the diagonal of $\boldsymbol{E_4}$ which is therefore zero on the off-diagonal. Computing these values we have, $m,m' \in C_{j+1}^{\cos+} \implies \tfrac{m}{2},\tfrac{m'}{2} \in C_{j+1}^{\sin+}$ and $2m,2m' \notin C_j$. Therefore only the scales $j$ and $j+1$ are used in the summation. Consider these diagonal elements of $\boldsymbol{E_4}$ with,
\begin{align*}
	E_4 &= \1{\left\{ m,m' \in C_{j+1}^{\cos+} : m=m'\right\}}\left\{ \psi_{m2^{-j}} \overline{\psi_{m'2^{-j}}} + \psi_{m2^{-(j+1)}} \overline{\psi_{m'2^{-(j+1)}}}  \right\}\\
		&= \1{\left\{ m=m'\right\}} e^{i \pi(m-m')2^{-j}}\cos\left( \frac{\pi}{2} \nu(\tfrac{3}{2}|m2^{-j}|-1)\right)  \cos\left( \frac{\pi}{2} \nu(\tfrac{3}{2}|m'2^{-j}|-1)\right) \\
		&\quad + \1{\left\{ m=m'\right\}} e^{i \pi(m-m')2^{-(j+1)}}\sin\left( \frac{\pi}{2} \nu(3|m2^{-(j+1)}|-1)  \right)\sin\left( \frac{\pi}{2} \nu(3|m'2^{-(j+1)}|-1) \right)\\
		&= 1
\end{align*}
since $\sin^2\theta + \cos^2 \theta = 1$ for all $\theta \in \mathbb{R}$.

{\sl Case $\boldsymbol{E_2}$ and $\boldsymbol{E_3}$}\\
Lastly, consider the case for $\boldsymbol{E_3}$, the case for $\boldsymbol{E_2}$ follows by symmetry. In this context, $m\in C_{j+1}^{\cos +}$ and $m'\in C_{j+1}^{\cos -}$. Consider the differences $m-m'$ along the main diagonal of $\boldsymbol{E_3}$ which are identical since the values of $m\in C_{j+1}^{\cos +}$ and $m'\in C_{j+1}^{\cos -}$ are consecutive. The first diagonal element is when $m = \lceil \tfrac{2^{j+1}}{3} \rceil$ and $m' =-\lfloor \tfrac{2^{j+2}}{3} \rfloor$ yielding,
\begin{align*}
	 m-m' &= \lceil \tfrac{2^{j+1}}{3} \rceil+\lfloor \tfrac{2^{j+2}}{3} \rfloor \\
		&= \tfrac{2^{j+1}}{3} +1 - \left\{\tfrac{2^{j+1}}{3} \right\}+\tfrac{2^{j+2}}{3} -\left\{ \tfrac{2^{j+2}}{3}\right\}\\
		&= 2^{j+1} - \left\{\tfrac{2^{j+1}}{3} \right\}  -\left\{ \tfrac{2^{j+2}}{3}\right\}+ 1\\
		&= 2^{j+1}  - \left\{ \tfrac{2^{j+1}}{3} \right\} - \left\{ 2^{j+1}-\tfrac{2^{j+1}}{3} \right\}+1\\
		&=2^{j+1}-1 +1 = 2^{j+1}.
\end{align*}
Similarly the maximum and minimum distances are, $\max_{m \in C_{j+1}^{\cos+},m' \in C_{j+1}^{\cos-}}(m-m')= 2\lfloor \tfrac{2^{j+2}}{3} \rfloor$ and $\min_{m \in C_{j+1}^{\cos+},m' \in C_{j+1}^{\cos-}}(m-m')= 2\lceil \tfrac{2^{j+1}}{3} \rceil$. Therefore the range of possible distances between $m$ and $m'$ are of length $2 \left(\lfloor \tfrac{2^{j+2}}{3} \rfloor -  \lceil \tfrac{2^{j+1}}{3} \rceil\right) =\frac{2^{j+2}-2(-1)^{j+1}}{3} - 2 < 2^{j+1}$. Therefore $m-m' = s2^j$ for some $s \in \mathbb{Z}$ only on the diagonal and $s = 2$ in this case ($m = m' + 2^{j+1}$). Thus again,  $\boldsymbol{E_3}$ is a diagonal matrix. Computing these values we have, $m \in C_{j+1}^{\cos+} \implies \tfrac{m}{2} \in C_{j+1}^{\sin+}$ and $2m \notin C_j$. Similar cases apply to $m'$. Therefore only the scales $j$ and $j+1$ are used in the summation. Consider these diagonal elements of $\boldsymbol{E_3}$,
\begin{align*}
	E_3 &= \1{\left\{ m \in C_{j+1}^{\cos+} ,m' \in C_{j+1}^{\cos-}: m=m'+2^{j+1}\right\}} \left( \psi_{m2^{-j}} \overline{\psi_{m'2^{-j}}} + \psi_{m2^{-(j+1)}} \overline{\psi_{m'2^{-(j+1)}}}  \right)\\
		&= \1{\left\{ m=m' + 2^{j+1}\right\}} e^{i \pi(m-m')2^{-j}}\cos\left( \frac{\pi}{2} \nu(\tfrac{3}{2}|m2^{-j}|-1)\right)  \cos\left( \frac{\pi}{2} \nu(\tfrac{3}{2}|m'2^{-j}|-1)\right) \\
		&\qquad + \1{\left\{m=m' + 2^{j+1}\right\}} e^{i \pi(m-m')2^{-{j-1}}}\sin\left( \frac{\pi}{2} \nu(3|\tfrac{m}{2}2^{-j}|-1)  \right)\sin\left( \frac{\pi}{2} \nu(3|\tfrac{m'}{2}2^{-j}|-1) \right)\\
		&= -\1{\left\{ m=m' + 2^{j+1}\right\}} \sin\left(  \frac{\pi}{2} \nu(\tfrac{3}{2}|m2^{-j}|-1)  \right)\sin\left(  \frac{\pi}{2} \nu(\tfrac{3}{2}|m'2^{-j}|-1) \right) \\
		&\qquad \qquad +\1{\left\{m=m' + 2^{j+1}\right\}}\cos\left(  \frac{\pi}{2} \nu(\tfrac{3}{2}|m2^{-j}|-1)\right)  \cos\left(  \frac{\pi}{2} \nu(\tfrac{3}{2}|m'2^{-j}|-1)\right)\\
		&=\1{\left\{m=m' + 2^{j+1}\right\}} \cos\left(  \frac{\pi}{2} \nu(\tfrac{3}{2}|m2^{-j}|-1)+ \frac{\pi}{2} \nu(\tfrac{3}{2}|m'2^{-j}|-1)\right).
\end{align*}
Again, exploit the fact that $\1{\left\{  m \in C_{j+1}^{\cos+} ,m' \in C_{j+1}^{\cos-} \right\}}$ implying $m > 0$ and $m' < 0$  along with the Meyer polynomial property in \eqref{eq:MeyerPoly} with the specific choice $x = \tfrac{3}{2}m2^{-j}-1$,
\begin{align*}
	E_3 &= \1{\left\{m=m' + 2^{j+1}\right\}} \cos\left(  \frac{\pi}{2} \nu(\tfrac{3}{2}|m2^{-j}|-1)+  \frac{\pi}{2} \nu(\tfrac{3}{2}|m'2^{-j}|-1)\right)\\
		&= \1{\left\{ -m'2^{-j}=2-m2^{-j}\right\}} \cos\left(  \frac{\pi}{2} \nu(\tfrac{3}{2}m2^{-j}-1)+  \frac{\pi}{2} \nu(-\tfrac{3}{2}m'2^{-j}-1)\right)\\
		&= \cos\left( \frac{\pi}{2} \nu(\tfrac{3}{2}m2^{-j}-1)+ \frac{\pi}{2} \nu(3-\tfrac{3}{2}m2^{-j}-1)\right)\\
		&= \cos \left( \frac{\pi}{2}\right) = 0.
\end{align*}
Therefore $\boldsymbol{E_3} = \boldsymbol{0}$ which completes the proof.
\end{proof}





\end{document}